\documentclass[11pt, leqno]{amsart}

\usepackage[mathscr]{eucal}
\usepackage{amsmath}
\usepackage{amsfonts,dsfont}
\usepackage{amsthm}
\usepackage{color}
\usepackage{placeins}

\theoremstyle{plain}
\newtheorem{theorem}{Theorem}[section]
\newtheorem{proposition}{Proposition}[section]
\newtheorem{corollary}{Corollary}[section]

\usepackage{a4wide}
\usepackage{amssymb}
\usepackage{graphicx}

\newtheorem{example}{Example}[section]

\newtheorem{remark}{Remark}[section]

\numberwithin{equation}{section}

\title[Magnetic curves]
{Magnetic curves in the real special linear group}
\author[J. Inoguchi]{Jun-ichi Inoguchi}
\address[J. Inoguchi]
{Institute of Mathematics,
University of Tsukuba,
1-1-1 Tennodai,
Tsukuba, 350-0006,
Japan}
\email{inoguchi@math.tsukuba.ac.jp}

\author[M.~I.~Munteanu]{Marian Ioan Munteanu}
\address[M.~I.~Munteanu]
{University 'Al. I. Cuza' of Iasi, 
Faculty of Mathematics, Bd. Carol I, no.~11,
700506 Iasi, Romania}
\email{marian.ioan.munteanu@gmail.com}

\date{\today}

\begin{document}

\begin{abstract}
We investigate contact magnetic curves in the real special linear group of degree $2$.
They are geodesics of the Hopf tubes over the projection curve.
We prove that periodic contact magnetic curves in $\mathrm{SL}_2\mathbb{R}$ can be quantized in the 
set of rational numbers.
Finally, we study contact homogeneous magnetic trajectories in $\mathrm{SL}_2\mathbb{R}$ and show that
they project to horocycles in $\mathbb{H}^2(-4)$.
\end{abstract}



\maketitle

\section{Introduction}

A magnetic field on a Riemannian manifold is defined by a closed 2-form. This
definition comes from the fact that a closed 2-form on a Riemannian manifold can be
regarded as a generalization of static magnetic fields on a Euclidean 3-space.
See e.g. \cite{jm:Com87, jm:Sun93}. A magnetic curve is a trajectory of a magnetic 
field and it is a solution of a second order differential equation known as the {\em Lorentz
equation} associated to the magnetic field. Lorentz equation generalizes the equation of 
geodesics under arc length parametrization. Hence, we may say that magnetic trajectories 
are perturbations of geodesics. 
On the other hand, magnetic curves derive also from the variational problem of the Landau-Hall
functional. See e.g. \cite{BRCF05}. In the absence of a magnetic field, this functional
is nothing but the kinetic energy functional. It is well known that geodesics are critical
points of the energy functional. This is another argument for saying that magnetic trajectories
are generalizations of geodesics. In this sense, the geometric properties of magnetic
curves show features of the underlying manifold, exactly how geodesics do.

The relation between geometry and magnetic fields have a long history. As is well 
known, the notion of linking number can be traced back to Gauss' work on terrestrial 
magnetism (see \cite{RN2011}). The linking number connects topology and 
Amp{\`e}re's law in magnetism. De Turck and Gluck studied magnetic curves and linking numbers in 
the $3$-sphere $\mathbb{S}^3$ and hyperbolic $3$-space $\mathbb{H}^3$ \cite{DeTG2008, DeTG2013}.

On the other hand, contact structures 
play a important role in $3$-dimensional topology.
In \cite{DIMN1}, we have studied magnetic trajectories 
in Sasakian manifolds with respect to the magnetic field 
derived from the contact structure (\textit{contact magnetic field}).
Even that we employ physical terms, when we study contact magnetic curves,
only we need is to get perturbations of geodesics obtained from the (almost)
contact structure on the manifold. For readers who are not familiar with 
magnetic fields, it is enough to consider that these trajectories are special
curves obtained as solutions of the Lorentz equation, which generalize the equation of geodesics.

In 2007, Taubes \cite{Taub07} proved the generalized Weinstein conjecture in dimension 3, namely,
{\it on a compact, orientable, contact $3$-manifold the Reeb vector field $\xi$ has at 
least one closed integral curve}. Linked to this problem it is important to investigate
the existence of periodic magnetic trajectories of the contact magnetic field defined by
$\xi$ in Sasakian manifolds, in particular in Sasakian space forms. 

In 2009, Cabrerizo et al. \cite{CFG09} have been looked for periodic orbits of the contact magnetic
field on the unit sphere $\mathbb{S}^3$. See also \cite{Bar08}. 
The present authors \cite{IM1} studied periodicity of 
contact magnetic trajectories on the $3$-dimensional Berger sphere equipped with 
the canonical homogeneous Sasakian structure of constant 
$\varphi$-sectional curvature $c>-3$. The Berger $3$-sphere 
$\mathcal{M}^{3}(c)$ equipped with Sasakian structure of 
constant $\varphi$-sectional curvature $c$ has the structure of a principal 
circle bundle. The base space of this fibering is the $2$-sphere $\mathbb{S}^{2}(c+3)$
of curvature $c+3$. This fibering includes the classical Hopf fibering 
$\mathbb{S}^{3}(1)\to\mathbb{S}^{2}(4)$.

There are three classes of $3$-dimensional simply connected 
Sasakian space forms of constant $\varphi$-sectional 
curvature $c${\rm:}
\begin{itemize}
\item the Berger $3$-sphere if $c>-3$; in particular the unit $3$-sphere  
$\mathbb{S}^3(1)$ if $c=1$;
\item the Heisenberg group $\mathrm{Nil}_3$ if $c=-3$;
\item the universal covering of $\mathrm{SL}_2\mathbb{R}$ if $c<-3$.
\end{itemize}
Note also that these spaces (with $c=1$, $c\leq -3$) are model spaces of Thurston geometries.
As a \textit{hyperbolic counterpart} of the Berger $3$-sphere, the 
special linear group $\mathrm{SL}_2\mathbb{R}$ admits a 
structure of principal circle bundle over the hyperbolic $2$-space 
$\mathbb{H}^{2}(c+3)$ of curvature $c+3<0$. 
The special linear group equipped with canonical left 
invariant Sasakian structure has constant $\varphi$-sectional curvature $c$.
This paper is a continuation of previous papers \cite{DIMN1, IM1}. 
Our aim is to study periodicity of contact magnetic trajectories of $\mathrm{SL}_2\mathbb{R}$.

This paper is structured as follows.
Firstly, considering the Hopf fibering from $\mathrm{SL}_2\mathbb{R}$ to $\mathbb{H}^2(-4)$,  
we show that contact magnetic curves in $\mathrm{SL}_2\mathbb{R}$ are geodesics of the Hopf tubes
over the projection curve. Then, we write the differential equations satisfied by the magnetic
trajectories in $\mathrm{SL}_2\mathbb{R}$. The key of this part is the use of Iwasawa decomposition.
In the following, we find a periodicity condition for contact magnetic curves in $\mathrm{SL}_2\mathbb{R}$.
We show that periodic magnetic curves in $\mathrm{SL}_2\mathbb{R}$ can be quantized in the 
set of rational numbers. 
Here we emphasize that periodic contact magnetic curves 
in $\mathrm{SL}_2\mathbb{R}$ are knots in a solid torus.
Knots in solid tori have been paid attention of knots researchers, see e.g., \cite{GM}. 
For torus knots in ideal magnetohydrodynamics, see \cite{OR}.
Periodic magnetic curves in $\mathrm{SL}_2\mathbb{R}$ provides nice examples of torus knots.

We pay a special attention on Legendre curves,
that is those curves whose contact angle is $\pi/2$.
It should be remarked that the notion of Legendre curve only depends 
on the contact structure of $\mathrm{SL}_2\mathbb{R}$.

Finally, we are also interested in the 
study of magnetic trajectories in $\mathrm{SL}_2\mathbb{R}$, which project to horocycles in
$\mathbb{H}^2(-4)$. Thus, in Section 6, we study homogeneous magnetic trajectories in
$\mathrm{SL}_2\mathbb{R}$, that is contact magnetic curves which are obtained from one-parameter
subgroups $\exp(tX)$, for $X\in\mathfrak{sl}_2\mathbb{R}$.

\section{Preliminaries}

\subsection{Magnetic curves} 
The motion of the charged particles in a 
Riemannian manifold under the action of 
the magnetic fields are known as magnetic curves. More precisely, 
a \textit{magnetic field} $F$ on a 
Riemannian manifold $(M, g)$ is a closed $2$-form $F$ and 
the \textit{Lorentz force} associated to $F$ 
is a tensor field $\phi$ of type $(1,1)$ such that
\begin{equation}\label{eqLforce}
F(X,Y) = g(\phi X,Y), \ \  X,Y \in \mathfrak{X}(M).
\end{equation}
A curve $\gamma$ on $M$ that satisfies the \textit{Lorentz equation} 
\begin{equation}
\label{eqL}
\nabla_{\dot\gamma}\dot\gamma=\phi(\dot\gamma),
\end{equation}
is called \textit{magnetic trajectory} of $F$ or simply a \emph{magnetic curve}.
Here $\nabla$ denotes the Levi-Civita connection associated to the metric $g$.
A magnetic field $F$ is said to be {\em uniform} if $\nabla F = 0$.  

It is well-known that the magnetic trajectories have constant speed. 
When the magnetic curve $\gamma(s)$ is arc length parametrized, 
it is called a \textit{normal magnetic curve}. 

The dimension $3$ is rather special, since it allows us to 
identify 2-forms with vector fields via the Hodge $\star$ operator 
and the volume form $dv_g$ of the (oriented) manifold. 
In this way, magnetic fields may be identified with divergence free 
vector fields by 
\begin{equation*}
\label{mag_FV}
F_{V}=\iota_{V}dv_{g}.
\end{equation*}
Magnetic fields $F$ corresponding to Killing vector fields are usually known as
\textit{Killing magnetic fields}. 
Their trajectories, called \textit{Killing magnetic curves}, are of great importance
since they are related to the Kirchhoff elastic rods. 
See \textit{e.g.}, \cite{BCFR07, BarrosRomero}.

\subsection{Sasakian manifolds} 
A {\em $(\varphi, \xi, \eta)$ structure} on 
a manifold $M$ is defined by a field $\varphi$ of endomorphisms of 
tangent spaces, a vector field $\xi$ and a 1-form $\eta$ satisfying
\begin{equation*}
\eta(\xi)=1,\ \  \varphi^{2}=-\mathrm{I}+\eta\otimes\xi, \ \ \varphi\xi=0,\ \ \eta \circ \varphi = 0.
\end{equation*}
If $(M, \varphi, \xi, \eta)$ admits a compatible Riemannian metric $g$, namely
\begin{equation*}
g(\varphi{X},\varphi{Y})=g(X,Y)-\eta(X)\eta(Y),
{\rm \ for\  all\  } X, Y\in \mathfrak{X}(M),
\end{equation*}
 then $M$ is said to have an 
{\em almost contact metric structure}, and $(M, \varphi, \xi, \eta, g)$ 
is called an {\em almost contact metric manifold}. 
Consequently, we have that $\xi$ is unitary
and $\eta(X) = g(\xi, X)$, for any $X\in\mathfrak{X}(M)$.

We define a 2-form $\Omega$ on $(M, \varphi, \xi, \eta, g)$ by 
\begin{equation}\label{2form}
\Omega(X,Y) = g(\varphi X, Y),
{\rm \ for\  all\  } X, Y\in \mathfrak{X}(M),
\end{equation}
called {\em the fundamental $2$-form} 
of the almost contact metric structure $(\varphi, \xi, \eta, g)$.

If $\Omega = d\eta$, 
then  $(M, \varphi, \xi, \eta, g)$ is called a {\em contact metric manifold}. Here $d\eta$ is defined by 
$
d\eta(X,Y) = \frac{1}{2}\big(X\eta(Y) - Y\eta(X)-\eta([X,Y])\big),
$
for any $X, Y\in \mathfrak{X}(M)$.
On a contact metric manifold $M$, the 1-form $\eta$ is a contact form (see Blair's book \cite{Blair}). 
The vector field $\xi$ is called the \textit{Reeb vector field} of $M$ and it is characterized by $\iota_\xi \eta =1$ and $\iota_\xi d\eta = 0$. 
Here $\iota$ denotes the interior product. 
In analytical mechanics, 
$\xi$ is traditionally called the 
\textit{characteristic vector field} of $M$.

An almost contact metric manifold $M$ is said to 
be {\em normal} if the normality tensor  
\linebreak
$
\displaystyle
S(X,Y)= N_\varphi(X, Y)+2d\eta(X,Y)\xi
$ vanishes, where $N_\varphi$ is the \textit{Nijenhuis torsion} of $\varphi$ defined by
$
\displaystyle
N_\varphi(X, Y) = 
[\varphi X,\varphi Y ] + \varphi^2[X, Y ] 
-\varphi[\varphi X, Y ] -\varphi[X, \varphi Y ],
$
for any $X, Y \in \mathfrak{X}(M)$.

A {\em Sasakian manifold} is defined as a normal contact metric manifold. 
Denoting by $\nabla$ the Levi-Civita connection associated to $g$, 
the Sasakian manifold $(M, \varphi, \xi, \eta, g)$ is characterized by 
\begin{equation*}
\label{Sasaki1}
(\nabla_X \varphi)Y = - g(X, Y)\xi + \eta(Y)X,
\ {\rm \ for\ any\ }\ X,Y\in\mathfrak{X}(M).
\end{equation*}
As a consequence, we have
\begin{equation}\label{Sasaki2}
\nabla_X \xi = \varphi X, \ \forall X\in \mathfrak{X}(M).
\end{equation}

A contact metric structure $(\varphi, \xi, \eta, g)$ is 
called {\em $K$-contact} if $\xi$ is a Killing vector field. 
Due to \eqref{Sasaki2} and the fact that $\varphi$ is skew-symmetric, 
it follows that a Sasakian manifold is $K$-contact. 
The converse is not true in general. 
Yet, a 3-dimensional manifold is 
Sasakian if and only if it is $K$-contact.

A plane section $\Pi$ at $p\in M^{2n+1}$ is called a $\varphi$-section 
if it is invariant under $\varphi_p$. The sectional curvature 
$K(\Pi)$ of a $\varphi$-section is called the 
{\em $\varphi$-sectional curvature} of $M^{2n+1}$ at $p$. 
A Sasakian manifold $(M^{2n+1}, \varphi, \xi, \eta, g)$ 
is said to be a {\em Sasakian space form} if it has constant $\varphi$-sectional curvature.

Take a positive constant $a$ and define a new 
Sasakian structure $(\varphi,\hat{\xi},\hat{\eta},\hat{g})$ on $M$ by
\[
\hat{\xi}:=\frac{1}{a}\xi,\ \
\hat{\eta}:=a\eta, \  \
\hat{g}:=ag+a(a-1)\eta\otimes\eta.
\]
This structure is called a $D$-\textit{homothetic deformation} of 
$(\varphi,\xi,\eta,g)$.
In particular, if $M(c)$ is a Sasakian space form 
of constant $\varphi$-sectional curvature $c$, 
then deforming the structure we obtain also a Sasakian space form $M(\hat{c})$, where 
$\hat{c} = \frac{c+3}{a}-3$. For every value of $c$ there exists Sasakian space forms, as follows: 
the elliptic Sasakian space forms, also known as the \textit{Berger spheres} if $c>-3$, 
the \textit{Heisenberg space} $\mathbb{R}^{2n+1}(-3)$, if $c=-3$, and $B^{2n}\times\mathbb{R}$ when $c<-3$. 
See also \cite[Theorem 7.15]{Blair}.
Note that the case $c>-3$ includes the standard unit sphere $\mathbb{S}^{2n+1}(1)$.

\begin{example}\label{Example2.1}{\rm
Let us identify the complex ball $B^2$ of curvature $-c^2\>(c>0)$ 
with the upper half plane 
\[
\mathbb{H}^{2}(-c^2)=\{(x,y)\in\mathbb{R}^2
\ \vert \ y>0\}
\]
equipped with the Poincar{\'e} metric 
$\bar{g}=(dx^2+dy^2)/(c^2y^2)$ of constant curvature $-c^2$. 
Then we have a global orthonormal frame field
\[
cy\frac{\partial}{\partial x},
\ \
cy\frac{\partial}{\partial y}.
\]
The standard complex structure $J$ of $\mathbb{H}^2(-4)$ is defined by 
\[
J\frac{\partial}{\partial x}
=
\frac{\partial}{\partial y},
\ \ 
J
\frac{\partial}{\partial y}
=-
\frac{\partial}{\partial x}.
\]
Then $\mathbb{H}^2(-c^2)=(\mathbb{H}^2(-c^2),J)$ 
is a K{\"a}hler manifold.

The K{\"a}hler form $\bar{\Omega}$ of 
$\mathbb{H}^2(-c^2)$ is defined by
\[
\bar{\Omega}(X,Y)=\bar{g}(JX,Y).
\]

Define the one-form $\omega$ on $\mathbb{H}^{2}(-c^2)$ by 
$\omega=2dx/(c^2y)$ then the K{\"a}hler form of 
$\mathbb{H}^{2}(-c^2)$ is 
$d\omega$. 


On the product manifold 
$\mathbb{H}^{2}(-c^2)\times\mathbb{R}$, we equip 
the contact metric structure $(\varphi,\xi,\eta,g)$ by
\[
\eta=dt+\pi^{*}\omega=dt+\frac{2dx}{c^2y},\ \ 
g=\pi^{*}g_B+\eta\otimes\eta
=\frac{dx^2+dy^2}{c^2y^2}+
\left(
dt+\frac{2dx}{c^2y}
\right)^2,
\ \
\xi=\frac{\partial}{\partial t},
\]
\[
\varphi \frac{\partial}{\partial x}=
\frac{\partial}{\partial y},\ \ 
\varphi \frac{\partial}{\partial y}=
-\frac{\partial}{\partial x}
+\frac{2}{c^2y}\frac{\partial}{\partial t},\ \
\varphi \frac{\partial}{\partial t}=0.
\]
Then the resulting 
Sasakian manifold is a simply connected 
Sasakian space form of constant 
$\varphi$-sectional curvature $-c^2-3$. 

Note that this Sasakian space form is the universal covering 
of $\mathrm{SL}_2\mathbb{R}$.
} 
\end{example}

\begin{remark}\label{Remark1}
{\rm
On the product manifold $\mathbb{H}^{2}
(-c^2)\times\mathbb{R}$,
we may consider the 
following one-parameter family of Riemannian metrics:
\[
g_{\nu}:=\frac{dx^2+dy^2}{c^2\>y^2}+
\left(dt+\frac{\nu\>dx}{c^2\>y}\right)^2,\ \ 
\nu\in\mathbb{R},\ \ \nu\geq 0.
\]
For $\nu=2$, we recover the 
metric of the Sasakian space form of constant 
$\varphi$-sectional 
curvature $-c^2-3$. On the other hand, for 
$\nu=0$, we obtain the \emph{Riemannian product} 
$\mathbb{H}^{2}(-c^2)\times\mathbb{E}^1$.
Moreover, when $\nu=c^2$, we obtain the \emph{Sasaki-lift metric} 
of the universal covering of the 
unit tangent sphere bundle $\mathrm{U}\mathbb{H}^{2}(-c^2)$ 
of $\mathbb{H}^{2}(-c^2)$.
}
\end{remark}

\subsection{Magnetic curves in Sasakian manifolds}

Let $(M, \varphi, \xi, \eta, g)$ be a contact metric manifold and let $\Omega$ be the fundamental 2-form defined by \eqref{2form}. 
Since $\Omega=d\eta$ on a contact metric manifold, $\Omega$ is a closed 2-form, thus 
we can define a magnetic field on $M$ by
\begin{equation*}
F_q(X,Y) = q\Omega(X,Y), 
\end{equation*}
where $X,Y\in\mathfrak{X}(M)$ and $q$ is a real constant.
We call $F_q$ the {\em contact magnetic field} with the {\em strength} $q$.
Notice that if $q=0$, then the contact magnetic field vanishes identically and the magnetic curves are the geodesics of $M$. 
In the sequel we assume $q\neq 0$.

The {\em Lorentz force} $\phi_q$ associated to the contact magnetic field $F_q$ may be easily determined combining 
\eqref{2form} and \eqref{eqLforce}, namely
\begin{equation*}
\phi_q = q\varphi, 
\end{equation*} 
where $\varphi$ is the field of endomorphisms of the contact metric structure. 

In this setting, the Lorentz equation \eqref{eqL} can be written as
\begin{equation}
\label{eqLcontact}
\nabla_{\gamma^\prime}\gamma^\prime= q\varphi\gamma^\prime,
\end{equation}
where $\gamma:I\subseteq\mathbb{R}\to M^{2n+1}$ is a smooth curve parametrized by arc length.
The solutions of \eqref{eqLcontact} are called {\em normal magnetic curves} or {\em trajectories} for $F_q$.

%

\subsection{Contact magnetic curves in $3$-dimensional Sasakian manifolds}
Now we assume that $M$ is a $3$-dimensional Sasakian manifold.
A curve $\gamma(u)$, parametrized by arclength, is said to be 
\textit{slant} if it makes constant angle with the Reeb vector flow,
that is the \emph{contact angle} $\sigma(u)$, defined by 
$\cos\sigma(u)=g(\gamma^{\prime}(u),\xi_{\gamma(u)})$, is constant along 
$\gamma$. 

\begin{proposition}[\cite{DIMN1}, \cite{IL}]
Let $M$ be a $3$-dimensional Sasakian manifold. 
Then a contact magnetic curve $\gamma(u)$, parametrized by arclength, is a slant helix with first curvature 
$\kappa_1=|q|\sin\sigma$ and second curvature 
$\kappa_2=|q\cos\sigma-1|$. 
The principal normal $N$ and binormal $B$ are given by 
\[
N=\frac{\varepsilon}{|\sin\sigma|}\varphi\gamma^{\prime},
\ \ 
B=\frac{\varepsilon}{|\sin\sigma|}(\xi-\cos\sigma\>\gamma^{\prime}).
\]
Here $\varepsilon$ is the signature of $q$.
\end{proposition}

\section{The special linear group} 
\subsection{Iwasawa decomposition}
Let 
$\mathrm{SL}_{2}\mathbb{R}$ be
the real special linear group of degree $2$:
\[
\mathrm{SL}_{2}\mathbb{R}=\left \{
\left (
\begin{array}{cc}
a & b \\ c & d
\end{array}
\right )
\ \biggr \vert \
a,b,c,d \in \mathbb{R},
\ ad-bc=1
\right \}.
\]
By using the Iwasawa decomposition 
$\mathrm{SL}_2\mathbb{R}=NAK$ of $\mathrm{SL}_2\mathbb{R}$; 
\begin{align*}
N&=\left \{
\left (
\begin{array}{cc}
1 & x \\ 0 & 1 
\end{array}
\right )
\
\biggr \vert \ 
x \in \mathbb{R}
\right \}, \qquad
\mbox{(Nilpotent \ part)}
\\
A&=
\left \{
\left (
\begin{array}{cc}
\sqrt{y} & 0 \\ 0 & 1/\sqrt{y}
\end{array}
\right )
\ \biggr
\vert \
y>0
\right \}, \qquad
\mbox{(Abelian \ part)}
\\
K&=
\left \{
\left (
\begin{array}{cc}
\cos \theta & \sin \theta \\
-\sin \theta & \cos \theta 
\end{array}
\right )
\ \biggr \vert 
\ 0\leq \theta < 2\pi
\right \}=\mathrm{SO}(2), \quad
\mbox{(Maximal \ torus)}
\end{align*}
we can introduce the following
global coordinate system $(x,y,\theta)$
of $\mathrm{SL}_2\mathbb{R}$:
\begin{equation} \label{coord}
(x,y,\theta)\longmapsto
\left (
\begin{array}{cc}
1 & x \\ 0 & 1 
\end{array}
\right )
\left (
\begin{array}{cc}
\sqrt{y} & 0 \\ 0 & 1/\sqrt{y}
\end{array}
\right )
\left (
\begin{array}{cc}
\cos \theta & \sin \theta \\
-\sin \theta & \cos \theta 
\end{array}
\right ).
\end{equation}
The mapping 
\[
\psi:\mathbb{H}^2(-4) \times 
\mathbb{S}^1 \rightarrow \mathrm{SL}_2\mathbb{R};\ \ 
\psi(x,y,\theta):=
\begin{pmatrix} 1 & x \\ 0 & 1 
\end{pmatrix}
\begin{pmatrix}
\sqrt{y} & 0 \\ 0 & 1/\sqrt{y} 
\end{pmatrix} 
\begin{pmatrix}
\cos \theta & \sin \theta \\ 
-\sin \theta & \cos \theta
\end{pmatrix}
\]
is a diffeomorphism onto 
$\mathrm{SL}_2\mathbb{R}$. 
Hereafter, we shall refer $(x,y,\theta)$ as a global
coordinate system of $\mathrm{SL}_2\mathbb{R}$.
Hence $\mathrm{SL}_2\mathbb{R}$ is diffeomorphic to
$\mathbb{R} \times \mathbb{R}^{+}\times\mathbb{S}^1$
and hence diffeomorphic to $\mathbb{R}^{3}\setminus \mathbb{R}$.
Since $\mathbb{R} \times \mathbb{R}^{+}$ is diffeomorphic
to open unit disk $\mathbb{D}$, $\mathrm{SL}_2\mathbb{R}$ is diffeomorphic
to open solid torus $\mathbb{D} \times \mathbb{S}^1$.

\begin{proposition}
The Iwasawa decomposition of an 
element $p=(p_{ij})\in\mathrm{SL}_{2}\mathbb{R}$ is given 
explicitly by $p=n(p)a(p)k(p)$, where 
\[
n(p)=\left(
\begin{array}{cc}
1 & x\\
0 &1
\end{array}
\right), \ \ 
a(p)=\left(
\begin{array}{cc}
\sqrt{y} & 0\\
0 &1/\sqrt{y}
\end{array}
\right), \ \ 
k(p)=\left(
\begin{array}{cc}
\cos\theta & \sin\theta\\
-\sin\theta &\cos\theta
\end{array}
\right)
\]
with 
\[
x=\frac{p_{11}p_{21}+p_{12}p_{22}}
{(p_{21})^2+(p_{22})^2},
\  \
y=\frac{1}{(p_{21})^2+(p_{22})^2},
\ \ 
e^{i\theta}=\frac{p_{22}-ip_{21}}{\sqrt{(p_{21})^2+(p_{22})^2}}.
\]
\end{proposition}

\subsection{}
As is well known, the Lie algebra $\mathfrak{sl}_2\mathbb{R}$ 
of $\mathrm{SL}_2\mathbb{R}$ is given explicitly by
\[
\mathfrak{sl}_2\mathbb{R}
=\left\{
X \in \mathrm{M}_{2}\mathbb{R}\ \biggr \vert
\ \mathrm{tr}\> X=0 \right \}. 
\]
We take the following basis of $\mathfrak{sl}_2\mathbb{R}$:
\[
E=\begin{pmatrix} 0 & 1 \\ 0 & 0 \end{pmatrix},\ \ 
F=\begin{pmatrix} 0 & 0 \\ 1 & 0 \end{pmatrix},\ \ 
H=\begin{pmatrix} 1 & 0 \\ 0 & -1 \end{pmatrix}.
\]
This basis satisfies the commutation relations:
\[
[E,F]=H,\ \ 
[F,H]=2F,\ \
[H,E]=2E.
\]

The Lie algebra $\mathfrak{n}$, 
$\mathfrak{a}$ and $\mathfrak{k}$ 
of closed groups $N$, $A$ and $K$ are given by 
\[
\mathfrak{n}=\mathbb{R}E,\ \ 
\mathfrak{a}=\mathbb{R}H,\ \ 
\mathfrak{k}=\mathbb{R}(E-F).
\]

The left invariant vector fields obtained by left translating $E$, $F$ and $H$ 
are denoted by the same letter $E$, $F$ and $G$, 
respectively. These left invariant vector fields are given by
\begin{align*}
E&=
y\cos(2\theta)\frac{\partial}{\partial x}
+y\sin(2\theta)\frac{\partial}{\partial y}
+\sin^2\theta\frac{\partial}{\partial \theta},
\\
F&=
y\cos(2\theta)\frac{\partial}{\partial x}
+y\sin(2\theta)\frac{\partial}{\partial y}
-\cos^2\theta\frac{\partial}{\partial \theta},
\\
H&=
-2y\sin(2\theta)\frac{\partial}{\partial x}
+2y\cos(2\theta)\frac{\partial}{\partial y}
+\sin(2\theta)\frac{\partial}{\partial \theta}.
\end{align*}
One notices that  
\[
\frac{\partial}{\partial \theta}=E-F
\]
is left invariant. 
On the other hand we have
\begin{align*}
E+F
=&\cos(2\theta)
\left(
2y\frac{\partial}{\partial x}
-\frac{\partial}{\partial \theta}
\right)
+\sin(2\theta)\left(
2y\frac{\partial}{\partial y}
\right),
\\
H=&-
\sin(2\theta)
\left(
2y\frac{\partial}{\partial x}
-\frac{\partial}{\partial \theta}
\right)+\cos(2\theta)\left(
2y\frac{\partial}{\partial y}
\right).
\end{align*}
Here we introduce a frame field 
$\{\epsilon_1,\epsilon_2,\epsilon_3\}$ by 
\begin{equation}\label{framefield}
\epsilon_1=2y\frac{\partial}{\partial x}
-\frac{\partial}{\partial \theta},\  \
\epsilon_2=2y\frac{\partial}{\partial y},\  \
\epsilon_3=\frac{\partial}{\partial \theta}.
\end{equation}
This frame field is related to 
$\{E+F, H,E-F\}$ by
\[
(E+F,H,E-F)=(\epsilon_1,\epsilon_2,\epsilon_3)
\left(
\begin{array}{ccc}
\cos(2\theta) & -\sin(2\theta) &0\\
\sin(2\theta) &\cos(2\theta) &0\\
0 & 0 & 1
\end{array}
\right).
\]

\begin{remark}{\rm
The Lie algebra $\mathfrak{h}$ is the 
Cartan subalgebra of $\mathfrak{sl}_{2}\mathbb{R}$. 
Moreover $\mathfrak{n}$ and $\mathbb{R}F$ are 
root spaces with respect to $\mathfrak{h}$.
The decomposition 
$\mathfrak{sl}_{2}\mathbb{R}=\mathfrak{h}
\oplus \mathfrak{n}\oplus\mathbb{R}F$ is the root 
space decomposition (or Gauss decompostion) of 
$\mathfrak{sl}_{2}\mathbb{R}$.
}
\end{remark}

\subsection{}
The special linear group $\mathrm{SL}_2\mathbb{R}$ 
acts transitively and
isometrically on the upper half plane:
\[
\mathbb{H}^2(-4)=
\left(
\{(x,y)\in \mathbb{R}^2 \ | \ y>0\},
\frac{dx^2+dy^2}{4y^2}
\right)
\]
of constant curvature $-4$ by the 
linear fractional transformation as
\[
\left(
\begin{array}{cc}
a & b\\
c & d
\end{array}
\right)\cdot z=\frac{az+b}{cz+d}.
\]
Here we regard a point $(x,y)\in\mathbb{H}^2(-4)$ as a complex 
number $z=x+yi$.

\begin{remark}
{\rm A linear fractional transformation 
determined by a matrix 
$\displaystyle \left(
\begin{array}{cc}
a & b\\
c & d
\end{array}
\right)\not=\pm~\mathrm{Id}$ with $ad-bc=1$ is said to be 
\begin{itemize}
\item \emph{elliptic} if $|a+d|<2$;
\item \emph{parabolic} if $|a+d|=2$;
\item \emph{hyperbolic} if $|a+d|>2$.
\end{itemize}
}
\end{remark}

The isotropy subgroup of $\mathrm{SL}_2\mathbb{R}$ at $i=(0,1)$ is 
the rotation group $\mathrm{SO}(2)$.
The natural projection $\pi:(\mathrm{SL}_2\mathbb{R},g)
\to \mathrm{SL}_2\mathbb{R}/\mathrm{SO}(2)=\mathbb{H}^2(-4)$
is given explicitly by
\[
\pi(x,y,\theta)=(x,y) \in \mathbb{H}^2(-4)
\]
in terms of the global coordinate
system \eqref{coord}.

The tangent space $T_{i}\mathbb{H}^{2}(-4)$ at 
the origin $i=(0,1)$ is identified with the 
vector subspace $\mathfrak{m}$ defined by
\[
\mathfrak{m}=\{X\in\mathfrak{sl}_{2}\mathbb{R}\>|\>
{}^t\!X=X\}.
\]
The Lie algebra $\mathfrak{g}=\mathfrak{sl}_2\mathbb{R}$ has 
the orthogonal splitting $\mathfrak{g}=\mathfrak{k}
\oplus\mathfrak{m}$.
This splitting can be carried out explicitly as
\[
X=X_{\mathfrak k}+X_{\mathfrak m}, \ \ 
X_{\mathfrak k}=\frac{1}{2}(X-{}^t\!X),
\ \ 
X_{\mathfrak m}=\frac{1}{2}(X+{}^t\!X).
\]

\subsection{}

Define a one-parameter family of inner products
$\{\langle \cdot, \cdot\rangle_{\lambda} \  \vert \  \lambda 
\in \mathbb{R}^*\}$ on $\mathfrak{sl}_2\mathbb{R}$ 
so that 
$\{E,\ F,\ H/\lambda\}$
is orthonormal with respect to 
$\langle \cdot,\cdot \rangle_{\lambda}$.
By left-translating these inner products,
we equip a one parameter family 
$\{g_\lambda\}$
of left invariant Riemannian metrics 
on $\mathrm{SL}_2\mathbb{R}$. This family $\{g_\lambda\}_{\lambda\in\mathbb{R}^{*}}$
is different from the family $\{g_\nu\}_{\nu\geq 0}$ in Remark~\ref{Remark1}.
With respect to the global coordinate system $(x,y,\theta)$,
each $g_\lambda$ is expressed as
\[
\frac{1}{2y^2}
\begin{pmatrix}
2(\cos^{4}
\theta+\sin^{4}
\theta+\lambda^{2}
\sin^{2}\theta\cos^{2}\theta) &
(1-\lambda^{2}/2) \sin 2\theta \cos 2\theta &
2y \\
(1-\lambda^{2}/2) \sin 2\theta \cos 2\theta &
\sin^{2} 2\theta +(\lambda^{2}/2)\cos^{2} 2\theta &
0 \\
2y & 0 & 4y^2
\end{pmatrix}.
\]
In particular the $x$ and the $y$-coordinate curves are orthogonal
if and only if $\lambda=\pm\sqrt{2}$. 
\par
The left invariant metric $g_{\sqrt{2}}$ is given by
\[
g_{\sqrt{2}}=2\left\{
\frac{dx^2+dy^2}{4y^2}+
\left(d\theta+\frac{dx}{2y}
\right)^2
\right \}.
\]
Here we would like to remark that  one-forms 
\[
\frac{dx}{2y},\ \frac{dy}{2y},\ d\theta+\frac{dx}{2y}
\]
are globally defined on $\mathrm{SL}_2\mathbb{R}$.

\subsection{}
For simplicity we shall restrict our attention 
to $g_{\sqrt{2}}$.
In addition, to adapt our computations 
to Sasakian geometry, we use the following homothetical
change of $g_{\sqrt{2}}$.
\[
g:=\frac{1}{2}g_{\sqrt{2}}=
\frac{dx^2+dy^2}{4y^2}+
\left(
d\theta+\frac{dx}{2y}
\right)^2.
\]
It is easy to see that the projection 
$\pi:(\mathrm{SL}_2\mathbb{R},g)
\rightarrow \mathrm{SL}_2\mathbb{R}/\mathrm{SO}(2)
=\mathbb{H}^2(-4)$ is a Riemannian submersion 
with totally geodesic fibres. 
This submersion 
$\pi:(\mathrm{SL}_2\mathbb{R},g) \to \mathbb{H}^2(-4)$ 
is called the
{\em hyperbolic Hopf fibering} of $\mathbb{H}^2(-4)$.

\subsection{}
On the Lie algebra 
$\mathfrak{g}=\mathfrak{sl}_2\mathbb{R}$, the inner product 
$\langle\cdot,\cdot\rangle$ at the identity 
induced from 
$g$ is written as
\[
\langle X,Y \rangle=
\frac{1}{2} \ \mathrm{tr}\>
({}^{t}\!XY ),\ \ 
X,Y \in \mathfrak{sl}_{2}\mathbb{R}.
\]
By using this formula,
we can see that the metric $g$
is not only invariant by 
$\mathrm{SL}_2\mathbb{R}$-left
translation
but also right translations by 
$\mathrm{SO}(2)$.
Hence the Lie group $\mathrm{SL}_2\mathbb{R}
\times\mathrm{SO}(2)$ 
with multiplication:
\[
(a,b)(a^{\prime},b^{\prime})=(aa^{\prime},bb^{\prime})
\]
acts isometrically on $\mathrm{SL}_2\mathbb{R}$ via the 
action:
\[
(\mathrm{SL}_2\mathbb{R}
\times\mathrm{SO}(2))\times\mathrm{SL}_2\mathbb{R}
\rightarrow \mathrm{SL}_2\mathbb{R};\ \ \ 
(a,b) \cdot X=aXb^{-1}.
\] 
Furthermore, this action of $\mathrm{SL}_2\mathbb{R}
\times\mathrm{SO}(2)$ on $\mathrm{SL}_2\mathbb{R}$ is transitive,
hence $\mathrm{SL}_2\mathbb{R}$ is a Riemannian homogeneous space
of $\mathrm{SL}_2\mathbb{R}
\times\mathrm{SO}(2)$. 
The isotropy subgroup of 
$\mathrm{SL}_2\mathbb{R}
\times\mathrm{SO}(2)$ at the identity matrix  $\mathrm{Id}$ is 
the diagonal subgroup
\[
\Delta K=\{(k,k)\>\vert\> k\in K\}\cong K
\]
of $K\times K$. 
The coset space $(\mathrm{SL}_2\mathbb{R}
\times\mathrm{SO}(2))/\mathrm{SO}(2)$
is a naturally reductive homogeneous
space. 

The tangent space $T_{\mathrm{Id}}\mathrm{SL}_2\mathbb{R}$ is the 
Lie algebra $\mathfrak{g}=\mathfrak{sl}_2\mathbb{R}$. This tangent space is 
identified with the vector subspace $\mathfrak{p}$ defined by
\[
\mathfrak{p}=\{(V+W,2W)\>\vert\>V\in\mathfrak{m},
\>
W\in\mathfrak{k}\}.
\]
The Lie algebra of the product group 
$G\times K$ is $\mathfrak{g}\oplus\mathfrak{k}$.
On the other hand the Lie algebra of $\Delta K$ is
\[
\Delta\mathfrak{k}=\{(W,W)\>\vert\>W\in\mathfrak{k}\}\cong 
\mathfrak{k}.
\]
The Lie algebra $\mathfrak{g}\oplus
\mathfrak{k}$ is 
decomposed as
$\mathfrak{g}\oplus\mathfrak{k}=\Delta(\mathfrak{k})
\oplus\mathfrak{p}$.

Every $(X,Y)\in\mathfrak{g}\oplus
\mathfrak{k}$ is decomposed as
\[
(X,Y)=(2X_{\mathfrak k}-Y,2X_{\mathfrak k}-Y)
+(X_{\mathfrak m}+(Y-X_{\mathfrak k}),
2(Y-X_{\mathfrak k})).
\]

\subsection{}
We choose an orthonormal basis of $(\mathfrak{sl}_2\mathbb{R},\langle
\cdot,\cdot \rangle)$ by
\[
E_{1}=\sqrt{2}E,\ E_{2}=\sqrt{2}F,\ E_3=H.
\]
Then the commutation relations are 
\[
[E_1,E_2]=2E_3,\ [E_2,E_3]=2E_2,\ [E_3,E_1]=2E_1.
\]
Let us denote the Levi-Civita connection of
$(\mathrm{SL}_2\mathbb{R},g)$ by $\nabla$.
By using the Koszul formula:
\[
2\langle \nabla_{X}Y,Z\rangle=
-\langle X,[Y,Z]\rangle
+\langle Y,[Z,X]\rangle
+\langle Z,[X,Y]\rangle,\ \ 
X,Y,Z \in {\frak g},
\]
one can obtain the following formulas:
\begin{align*}
& \nabla_{E_1}E_{1}=2E_{3},\ \nabla_{E_1}E_{2}=E_{3},\ 
\nabla_{E_1}E_{3}=-2E_{1}-E_{2}, 
\\
& \nabla_{E_2}E_{1}=-E_{3},\ \nabla_{E_2}E_{2}=-2E_{3},\ 
\nabla_{E_2}E_{3}=E_{1}+2E_{2}, 
 \\
&
\nabla_{E_3}E_{1}=-E_{2},\ \nabla_{E_3}E_{2}=E_{1},\ 
\nabla_{E_3}E_{3}=0. 
\end{align*}

The bi-invariance obstruction ${\mathsf{U}}$ defined by
\[
2\langle {\mathsf{U}}(X,Y),Z \rangle=
-\langle X,[Y,Z]\rangle+\langle Y,[Z,X] \rangle,
\ X,Y \in {\frak g}
\]
is given by
\begin{align}\label{U-tensor}
&{\mathsf{U}}(E_1,E_1)=2E_3,\ {\mathsf{U}}(E_1,E_2)=0,\ {\mathsf{U}}(E_1,E_3)=-E_1-E_2,
\\
&{\mathsf{U}}(E_2,E_2)=-2E_3,\ {\mathsf{U}}(E_2,E_3)=E_1+E_2,\ {\mathsf{U}}(E_3,E_3)=0.
\nonumber
\end{align}
The Levi-Civita connection is 
rewritten as 
\begin{equation}\label{LC}
\nabla_{X}Y=\frac{1}{2}[X,Y]+{\mathsf{U}}(X,Y),
\ \  X,Y 
\in \mathfrak{g}.
\end{equation}

\subsection{}
We take the following orthonormal coframe field
of $\mathrm{SL}_2\mathbb{R}$:
\begin{equation*}
\omega^1=\frac{dx}{2y},\ \
\omega^{2}=\frac{dy}{2y}, \ \
\omega^{3}=d\theta+\frac{dx}{2y}.
\end{equation*}
The dual frame field 
of $\{\omega^1,\omega^2,\omega^3\}$ is the frame field 
$\{\epsilon_1,\epsilon_2,\epsilon_3\}$ introduced by \eqref{framefield}.
Note that this orthonormal frame 
field is \textit{not} left invariant with respect to 
the Lie group structure.

The Levi-Civita connection $\nabla$ 
of $g$ is given by the following formulas:
\[
\nabla_{\epsilon_1}\epsilon_1=2\epsilon_2,\ \
\nabla_{\epsilon_1}\epsilon_2=-2\epsilon_1-\epsilon_3
,\ \
\nabla_{\epsilon_1}\epsilon_3=\epsilon_2,
\]
\begin{equation*}
\nabla_{\epsilon_2}\epsilon_1=\epsilon_3,\ \
\nabla_{\epsilon_2}\epsilon_2=0
,\ \
\nabla_{\epsilon_2}\epsilon_3=-\epsilon_1,
\end{equation*}
\[
\nabla_{\epsilon_3}\epsilon_1=\epsilon_2,\ \
\nabla_{\epsilon_3}\epsilon_2=- \epsilon_1
,\ \
\nabla_{\epsilon_3}\epsilon_3=0.
\]
The commutation relations
of the basis are given by
\begin{equation*}
[\epsilon_{1},\epsilon_{2}]=
-2\epsilon_{1}-2\epsilon_{3},\ \
[\epsilon_{1},\epsilon_{3}]=0,\ \
[\epsilon_{2},\epsilon_{3}]=0.
\end{equation*}
The Riemannian curvature tensor
$R$ of the metric $g$
defined by
\[
R(X,Y):=\nabla_{X}\nabla_{Y}-
\nabla_{Y}\nabla_{X}-
\nabla_{[X,Y]},\ \ 
X,Y\in \mathfrak{X}(\mathrm{SL}_2\mathbb{R})
\]
is 
described by the
following formulas:
\begin{equation*}
\begin{array}{cc}
  {R}(\epsilon_{1},\epsilon_{2})
  \epsilon_{1}=7\epsilon_{2}, &
 {R}(\epsilon_{1},\epsilon_{2})\epsilon_{2}
 =-7\epsilon_{1}, \\
 {R}(\epsilon_{1},\epsilon_{3})\epsilon_{1}
 =-\epsilon_{3}, &
 {R}(\epsilon_{1},\epsilon_{3})
 \epsilon_{3}=\epsilon_{1}, \\
 {R}(\epsilon_{2},\epsilon_{3})
 \epsilon_{2}=-\epsilon_{3}, &
{R}(\epsilon_{2},\epsilon_{3})
\epsilon_{3}=\epsilon_{2}.
\end{array}
\end{equation*}
The other significant components are zero.
\subsection{}
The one-form $\eta=d \theta+dx/(2y)$ 
is a {\em contact form} on
$\mathrm{SL}_2\mathbb{R}$, {\em i.e.},
$d \eta \wedge \eta \not=0$. The Reeb vector field 
of $\eta$ is $\xi=\epsilon_3$.
The contact distribution 
determined by $\eta$ 
coincides the 
horizontal distribution 
of the Riemannian 
submersion $\pi:G\to\mathbb{H}^2(-4)$.
\begin{remark}
{\rm
Under the identification 
$\mathfrak{k}\cong\mathbb{R}$, 
the contact form $\eta$ is regarded as a 
connection form of the principal 
circle bundle $\mathrm{SL}_2\mathbb{R}
\to\mathbb{H}^2(-4)$.
}
\end{remark}

Let us define an endomorphism field 
$\varphi$ by
\[
\varphi\> \epsilon_1=\epsilon_2,\ 
\varphi\> \epsilon_2=-\epsilon_1,\ 
\varphi\> \epsilon_3=0.
\]
Then
$(\varphi,\xi,\eta,g)$ satisfies the 
following relations:
\[
\varphi^2=-I+\eta \otimes \xi,
\ \
d \eta(X,Y)=g(\varphi{X},Y),
\]
\[
g(\varphi{X},\varphi{Y})=g(X,Y)-\eta(X)\eta(Y),
\]
\[
\nabla_{X}\xi=\varphi{X},
\]
\[
(\nabla_{X}\varphi)Y=-g(X,Y)\xi+\eta(Y)X
\]
for all $X$, $Y \in \mathfrak{X}(\mathrm{SL}_2\mathbb{R})$.
Thus the
structure $(\varphi,\xi,\eta)$ is 
an almost contact structure 
compatible to the metric $g$. 
In other words, structure
$(\varphi,\xi,\eta,g) $ is an
almost contact metric structure
associated to the the contact manifold
$(\mathrm{SL}_2\mathbb{R},\eta)$ \cite{IKOS3}.
Since all the structure tensor fields 
are left invariant, the resulting almost contact 
metric manifold
$(\mathrm{SL}_2\mathbb{R},\varphi,\xi,\eta,g)$ is a homogeneous
Sasakian manifold of constant $\varphi$-sectional curvature $-7$. 
The structure $(\varphi,\xi,\eta,g)$ is called the
{\em canonical Sasakian structure} 
of $\mathrm{SL}_2\mathbb{R}$. 

\begin{remark}
{\rm 
The Riemannian curvature tensor
$R$ of $(\mathrm{SL}_2\mathbb{R},g)$
is given 
explicitly by 
\begin{align*}
R(X,Y)Z  = & -
g(Y,Z)X+g(Z,X)Y \\
&  -2\>\{
\eta(Z)\eta(X)Y
-\eta(Y)\eta(Z)X \\
&+g(Z,X)\eta(Y)\xi
-g(Y,Z)\eta(X)\xi \\
	&-g(Y,\varphi{Z})\varphi{X}-g(Z,\varphi{X})\varphi{Y}+
2g(X,\varphi{Y})\varphi{Z}
\>\}
\end{align*}
in terms of the canonical Sasakian structure.

For more informations on 
the canonical Sasakian structure  
of $\mathrm{SL}_2\mathbb{R}$, we refer to \cite{IKOS3}.
}
\end{remark}
\section{Hopf tubes}
\subsection{Fundamental equations of Hopf tubes} \
Let $\pi:(SL_2\mathbb{R},g)\rightarrow\mathbb({H}^2(-4),\bar{g})$ be the hyperbolic Hopf 
fibering. Consider a regular curve $\beta:{\mathbb{R}}\longrightarrow 
{\mathbb{H}}^2(-4)$,
$u\mapsto \beta(u)$. As usual, 
$\beta$ will be parametrized by the arc length and 
let $\hat\beta$ be a horizontal
lift of $\beta$. This means that 
$\pi(\hat\beta(u))=\beta(u)$ for all $u\in{\mathbb{R}}$ and
$\langle\hat\beta(u)^{-1}\hat\beta'(u),\epsilon_3\rangle=0$.
If we represent $\beta(u)$ as 
$\beta(u)=(x(u)),y(u))$, then the horizontal lift
$\hat{\beta}(u)$ is given by 
$\hat{\beta}(u)=(x(u),y(u),\theta(u))$ whose third coordinate
$\theta(u)$ is determined by the ordinary differential equation
\[
\frac{d\theta}{du}=-\frac{1}{y(u)}\frac{dx}{du}
\]
with initial condition $\theta(0)=\theta_0$. 

The complete lift of $\beta$, namely $\pi^{-1}(\beta)$ is a flat surface 
in ${\mathrm{SL}}_2\mathbb{R}$ and it is usually called 
the \emph{Hopf tube} over $\beta$. 

Denote $\pi^{-1}(\beta)$ by $H_\beta$. 
The Hopf tube $H_{\beta}$ is represented as an 
immersion:
\[
F:{\mathbb{R}}\times{\mathbb{R}}\longrightarrow {\mathrm{SL}}_2\mathbb{R},\quad
(t,u)\longmapsto F(t,u)=\hat\beta(u) k(t),
\]
where 
\[
k(t)=\left(
\begin{array}{cc}
\cos t & \sin t\\
-\sin t &\cos t
\end{array}
\right).
\]
In other words,
\[
F(t,u)=(x(u),y(u),\theta(u)+t).
\]
The derivatives are 
\[
F_{u}=\frac{x^{\prime}}{2y}\epsilon_{1}+
\frac{y^{\prime}}{2y}\epsilon_{2},\ \ 
F_{t}=\epsilon_{3}.
\]
Hence the induced metric $g_{H_\beta}$ is computed as
\begin{equation*}
g_{H_\beta}=dt^2+du^2.
\end{equation*}
Let us compute the second fundamental form of $H_{\beta}$.
Let $\{\overline{T}(u),\overline{N}(u)\}$ the 
Frenet frame field of $\beta(u)$. 
As usual
\[
\overline{T}(u)=\beta^{\prime}(u),
\ \ 
\overline{N}(u)=J\overline{T}(u).
\]
The signed curvature $\kappa_{\beta}$ is defined by
\[
\overline{\nabla}_{\beta^\prime}\overline{T}=\kappa_{\beta}\overline{N}.
\]
Let us denote by $\hat{T}$ the horizontal lift 
of $\overline{T}$. Then $T$ tangents to $H_{\beta}$.
Moreover $\{\hat{T},\xi\}$ is an orthonormal frame field 
of $H_{\beta}$. The horizontal lift 
$\hat{N}$ of $N$ is a unit normal vector field 
of $H_{\beta}$. 
One can check that $\hat{N}=\varphi\hat{T}$. 
The Gauss formula of $H_{\beta}$ is given by \cite[\S 1.3]{I2004CM}:
\[
\nabla_{\hat{T}}\hat{T}=(\kappa_{\beta}\circ \pi)\>\hat{N},
\ \
\nabla_{\hat{T}}\xi=\nabla_{\xi}\hat{T}=\hat{N} ,\ \ 
\nabla_{\xi}\xi=0.
\]
These formula imply that $H_{\beta}$ is flat and 
the second fundamental form $h$ derived from $\hat{N}$ is 
\[
h(\hat{T},\hat{T})=\kappa_{\beta}\circ\pi,\ \ 
h(\hat{T},\xi)=1,\ \ 
h(\xi,\xi)=0.
\]
Thus, the mean curvature of $H_{\beta}$ is 
$(\kappa_{\beta}\circ\pi)/2$.

Hence we have proved the following fact.
\begin{proposition}
If $\beta$ is a curve on ${\mathbb{H}}^2(-4)$ 
of length $L$, then the corresponding Hopf tube $H_\beta$
is isometric to ${\mathbb{S}}^1(1)\times[0,L]$, where ${\mathbb{S}}^1(1)$ is the unit circle endowed with the metric
$dt^2$. Moreover, its mean curvature in ${\mathrm{SL}}_2\mathbb{R}$ is $(\kappa_\beta\circ \pi)/2$, where
$\kappa_\beta$ is the signed curvature of $\beta$ in ${\mathbb{H}}^2(-4)$.
\end{proposition}

If $\beta$ is a closed curve, \emph{i.e.}, 
$\beta(u+L)=\beta(u)$ for all 
$u\in\mathbb{R}$, then the relation
$F(t,u)=\hat\beta(u)k(t)$ defines a covering of the $(t,u)$ plane onto an immersed torus in 
$\mathrm{SL}_2\mathbb{R}$, called the \emph{Hopf torus} corresponding to $\beta$. 

\begin{remark}
{\rm
The Hopf tube 
$H_{\beta}$ can be parametrized by 
the following immersion.  
\[
F_1(u,v)=(x(u),y(u),v),\ \ u\in\mathbb{R},\> v\in\mathbb{S}^1.
\] 
Then by using the Iwasawa decomposition, 
the Hopf tube over $\beta$ can be parametrized as an immersion 
$F_1:\mathbb{R}\times\mathbb{S}^1\to\mathrm{SL}_2\mathbb{R}$:
\begin{equation}
F_1(u,v)=
\left (
\begin{array}{cc}
1 & x(u) \\ 0 & 1 
\end{array}
\right )
\left (
\begin{array}{cc}
\sqrt{y(u)} & 0 \\ 0 & 1/\sqrt{y(u)}
\end{array}
\right )
\left (
\begin{array}{cc}
\cos v & \sin v \\
-\sin v & \cos v
\end{array}
\right).
\end{equation}
Under this parametization the induced metric 
is computed as:
\[
\left(dv+\frac{x^{\prime}(u)}{2y(u)}du\right)^2+du^2.
\]
One can check that the induced metric on $H_{\beta}$ is flat.
This parametization is used in \cite{I2004, Kokubu}.
}
\end{remark}

For later use we recall here the classification of 
Hopf tubes with constant mean curvature \cite{Kokubu} (see also 
Appendix \ref{sectionA2}).

\begin{proposition}
[Classification 
of CMC Hopf tubes]
Let $\beta$ be a unit speed curve in $\mathbb{H}^2(-4)$ 
with curvature $\kappa$ and
$H_{\beta}$ the Hopf tube over $\beta$ in $\mathrm{SL}_2\mathbb{R}$.
Then $H_{\beta}$ is
of constant mean curvature if and only if
$\beta$ is a Riemannian circle in $\mathbb{H}^2(-4)$. 

The Hopf cylinder $H_{\beta}$ is
classified in the 
following way{\rm :}
\begin{enumerate}
\item If $\kappa_{\beta}=0$, then $H_\beta$ is a minimal Hopf tube over a geodesic.
\item 
If $0<\kappa_{\beta}^2<4$, then 
$H_{\beta}$ is a Hopf tube over an open circle 
or a Hopf tube over a line segment
$y=\pm (\sqrt{1-4\kappa^2}/(2\kappa))x$.

\item If $\kappa_{\beta}^2=4$, then
$H_{\beta}$ is a Hopf tube over a horocycle
or a Hopf tube over $y=\mbox{constant}$. 
\item If $\kappa_{\beta}^2>4$, then 
$H_{\beta}$ is a Hopf tube over a closed circle. In this case, 
$H_{\beta}$ is an embedded Hopf torus.
\end{enumerate}
\end{proposition}

\subsection{Contact magnetic curves and Hopf tubes}
We investigate the projection image of contact magnetic curves. 
First we recall the following fundamental fact.

\begin{proposition}
Let $\gamma(u)$ be an arclength parametrized contact magnetic curve 
in a $3$-dimensional Sasakian manifold $(M,\varphi,\xi,\eta,g)$. Then, 
the contact angle $\sigma(u)$, defined by $\cos\sigma(u)=g(\xi,\gamma^{\prime}(u))$,
is constant along $\gamma$.
\end{proposition}

Contact magnetic curves are included in some Hopf tubes. Moreover, we have the following.
\begin{theorem}
A contact magnetic curve $\gamma$ in 
$\mathrm{SL}_2\mathbb{R}$ is a 
geodesic of the Hopf tube $H_\beta$ over 
$\beta=\pi\circ \gamma$.
\end{theorem}
\proof
Let $\gamma(u)$ be a contact magnetic curve with 
strength $q$ parametrized by arclength $u$. 
Set $\beta=\pi\circ\gamma$ then $\gamma$ is 
contained in the Hopf tube $H_{\beta}$. 
Remark that $u$ is not, in general, 
the arclength parameter of $\beta$. 
Now let $\hat{N}$ the unit normal vector field along 
$H_{\beta}$ as before, then we have 
$\pi_{*}\hat{N}=\overline{N}$. 
The Gauss-formula of $H_{\beta}$ implies
\begin{equation}
\label{HopftubeGauss}
\nabla_{\gamma^{\prime}}
\gamma^{\prime}=\dot{\nabla}
_{\gamma^{\prime}}
\gamma^{\prime}+h(\gamma^{\prime},\gamma^{\prime})\hat{N}.
\end{equation}
Here $\dot{\nabla}$ is the Levi-Civita connection of 
$H_{\beta}$ and $h$ is as before.

Let us express $\gamma$ as $\gamma(u)=(x(u),y(u),\theta(u))$, then 
the velocity vector field $\gamma^{\prime}$ is 
given by
\begin{equation}
\label{magneticvelocity}
\gamma^{\prime}=
\frac{x^\prime}{2y}\epsilon_1
+
\frac{y^\prime}{2y}\epsilon_2
+\eta(\gamma^\prime)\epsilon_3, \ \ 
\eta(\gamma^\prime)
=\theta^\prime+
\frac{x^\prime}{2y}.
\end{equation}
Thus we get
$$
\varphi\gamma^{\prime}=
-\frac{y^\prime}{2y}\epsilon_1
+
\frac{x^\prime}{2y}\epsilon_2.
$$
On the other hand $\hat{N}$ is expressed as 
\[
\hat{N}=\frac{1}{\sqrt{(x^\prime)^2+(y^\prime)^2}}
\left(
-y^{\prime}\epsilon_1+x^{\prime}\epsilon_2
\right).
\]
It follows that $q\varphi\gamma^{\prime}$
is collinear to $\nu$. Comparing with 
the magnetic equation 
$\nabla_{\gamma^\prime}\gamma^{\prime}=q\varphi
\gamma^{\prime}$ and 
\eqref{HopftubeGauss} we find $\dot{\nabla}_{\gamma^\prime}
\gamma^{\prime}=0$.
\endproof

Let $\gamma(u)$ be an arclength parametrized contact magnetic curve 
in $\mathrm{SL}_2\mathbb{R}$, then the projection curve $\beta(u)=\pi(\gamma(u))$ has 
the velocity $\beta^{\prime}(u)=\pi_{*}\gamma^{\prime}$. Since $\pi$ is a 
Riemannian submersion, we have 
\[
|\beta^{\prime}(u)|^{2}=|\gamma^{\prime}(u)|^2-\eta(\gamma^{\prime})^2
=\sin^2\sigma.
\] 

\subsection{K{\"a}hler magnetic curves}

Let us consider the magnetic curve in $\mathbb{H}^2(-4)$ with 
respect to the magnetic field $\bar{F}_q:=\bar{q}\bar{\Omega}$. 
Here $\bar{\Omega}$ is the 
K{\"a}hler form of $\mathbb{H}^2(-4)$ defined by 
\linebreak
$\bar{\Omega}=(dx\wedge dy)/(2y^2)$ as in Example \ref{Example2.1}. 
The corresponding Lorentz equation is 
$\overline{\nabla}_{\beta^\prime}\beta^{\prime}=\bar{q}J\beta^\prime$.
Here $\overline{\nabla}$ is 
the Levi-Civita connection of 
$\mathbb{H}^{2}(-4)$. 
Comparing the Lorentz equation with the Frenet equation, we 
obtain that $\beta$ is a Riemannian circle in $\mathbb{H}^{2}
(-4)$ satisfying $\bar{q}=\kappa_{\beta}$. 
Hence normal magnetic trajectories 
are closed if and only if $|q|>2$.

By the fundamental equation of Riemannian submersion one can 
check the following result.
 
\begin{proposition}
The projection image $\beta(u)=\pi(\gamma(u))$ of a 
contact magnetic curve is a K{\"a}hler magnetic curve in 
$\mathbb{H}^{2}(-4)$. More precisely, $\beta$ satisfies the Lorentz equation
\linebreak
$\overline{\nabla}_{\beta^{\prime}}\beta^{\prime}=(q-2\cos\sigma)J\beta^{\prime}$. 
Hence $\gamma(u)$ is a geodesic in a Hopf tube over a Riemannian circle.
\end{proposition}

\section{Magnetic trajectories in $\mathrm{SL}_2\mathbb{R}$}

In our previous paper \cite{DIMN1},
we have proved that the classification of contact magnetic curves in 
Sasakian space forms of arbitrary dimension reduces to that in $3$-dimensional Sasakian space forms. 
More precisely, we have the following results.
\begin{theorem}\label{magnSasaki}
Let $(M^{2n+1},\varphi,\xi,\eta,g)$ be a Sasakian manifold and consider
$F_q$, $q\neq 0$, the contact magnetic field on $M^{2n+1}$. Then $\gamma$ is a normal
magnetic curve associated to $F_q$ in $M^{2n+1}$ if 
and only if $\gamma$ belongs to the following list:
\begin{itemize}
\item[{\rm a)}] geodesics, 
obtained as integral curves of $\xi${\rm;}
\item[{\rm b)}] non-geodesic 
$\varphi$-circles of curvature 
$\kappa_1 = \sqrt{q^2-1}$, for $|q|>1$, 
and of constant contact angle $\sigma = \arccos \frac{1}{q}${\rm;}
\item[{\rm c)}] Legendre $\varphi$-curves in $M^{2n+1}$ with curvatures 
$\kappa_1 = |q|$ and $\kappa_2=1$, i.e. $1$-dimensional integral
submanifolds of the contact distribution{\rm;}
\item[{\rm d)}] 
$\varphi$-helices of order $3$ 
with axis $\xi$, having curvatures 
$\kappa_1=|q|\sin\sigma$ and 
$\kappa_2=|q\cos\sigma-1|$, where $\sigma\neq\frac{\pi}{2}$ is the constant contact angle.
\end{itemize}
\end{theorem}

\begin{theorem}[\cite{DIMN1}]
\label{thm_b2nxr}
Let  $\gamma: I\subseteq \mathbb{R} 
\to B^{2n}(-4)\times\mathbb{R}$, 
be a smooth curve parametrized by arclength $s$ 
and let $F_q=q\Omega$, 
$q\neq 0$ be the contact 
magnetic field.
Then $\gamma$ is a normal magnetic 
curve associated to $F_q$ if and only if 
it belongs to the following list{\rm:} 
\begin{itemize}
\item[{\rm a)}] a geodesic obtained as 
integral curve of $\xi${\rm;}
\item[{\rm b)}] 
the horizontal lift of a magnetic trajectory in $B^2(-4)$ 
corresponding to the K{\"a}hler magnetic field $\bar{F} = \bar{q}\bar{\Omega}${\rm;}
\item[{\rm c)}] a helix in the 
$3$-dimensional Sasakian space form $\widetilde{{\rm PSL}}_2\mathbb{R}$ 
identified with $B^2(-4)\times\mathbb{R}$ 
as totally geodesic submanifold in $B^{2n}(-4)\times\mathbb{R}$. 
Moreover, $\gamma$ is a geodesic on a 
Hopf tube over a curve of constant curvature 
in $\mathbb{H}^2(-4)$. 
\end{itemize}
\end{theorem}

Let us now take a contact magnetic curve 
$\gamma(s)=(x(s),y(s),\theta(s))$ in $\mathrm{SL}_2\mathbb{R}$. 
Then the velocity vector field is given by 
\eqref{magneticvelocity}.

The acceleration vector field is computed as
\begin{align*}
\nabla_{\gamma^{\prime}}\gamma^{\prime}=
& \left \{\frac{x^{\prime \prime}y-x^{\prime}y^{\prime}}{2y^2}
-\frac{x^{\prime}y^{\prime}}{2y^2}
-\frac{y^{\prime}}{y}
\eta(\gamma^{\prime})
\right \}
\epsilon_1  +\left \{
\frac{y^{\prime \prime}y-(y^{\prime})^2}{2y^2}
+\frac{(x^{\prime})^2}{2y^2}
+\frac{x^{\prime}}{y}
\eta(\gamma^{\prime})
\right \}\epsilon_2 \\
& +\{\eta(\gamma^{\prime})\}^{\prime}
\epsilon_3.
\end{align*}
The the magnetic equation 
$\nabla_{\gamma^\prime}\gamma^\prime=q\varphi\gamma^\prime$ 
with strength $q$ is the following system:
\begin{align}
\frac{x^{\prime \prime}y-x^{\prime}y^{\prime}}{2y^2}
-\frac{x^{\prime}y^{\prime}}{2y^2}
-\frac{y^{\prime}}{y}
\eta(\gamma^{\prime})
\nonumber
&=
-\frac{qy^{\prime}}{2y},
\\
\frac{y^{\prime \prime}y-(y^{\prime})^2}{2y^2}
+\frac{(x^{\prime})^2}{2y^2}
+\frac{x^{\prime}}{y}
\eta(\gamma^{\prime})
&=\frac{qx^{\prime}}{2y},
\nonumber
\\
\{\eta(\gamma^\prime)\}^\prime
=\left[\theta^\prime+\frac{x^\prime}{2y}\right]^\prime&=0.
\nonumber
\end{align}
The third equation confirms that $\gamma$ is a 
\emph{slant curve}, that is $\gamma^{\prime}$ makes
constant angle $\sigma$ with the Reeb vector field $\xi$.  
By definition of $\sigma$, we notice that 
$\eta(\gamma^{\prime})=\cos\sigma \in [-1,1]$. 
The equations of magnetic trajectory become
\begin{align*}
& \frac{x^{\prime \prime}y
-x^{\prime}y^{\prime}}{2y^2}
-\frac{x^{\prime}y^{\prime}}{2y^2}
-(\cos\sigma)\frac{y^{\prime}}{y}=-\frac{qy^{\prime}}{2y},
\\
& \frac{y^{\prime \prime}y
-(y^{\prime})^2}{2y^2}
+\frac{(x^{\prime})^2}{2y^2}
+(\cos\sigma)\frac{x^{\prime}}{y}=\frac{qx^{\prime}}{2y}. 
\end{align*}
Put
\[
X=\frac{x^{\prime}}{2y},\
\ \ Y=\frac{y^{\prime}}{2y}.
\]
Then we have $X^2+Y^2+\cos^2\sigma=1$, which implies that $X^2+Y^2=\sin^2\sigma$.

Moreover, we find
\[
X^{\prime}=
\frac{x^{\prime \prime}y
-x^{\prime}y^{\prime}}{2y^2},
\ \ 
Y^{\prime}=\frac{y^{\prime \prime}y
-(y^{\prime})^2}{2y^2}.
\]
Hence, the equations of magnetic trajectory yield
the system 
\begin{equation}
\label{KM}
\left\{\begin{array}{l}
X^{\prime}-Y(2X+2\cos\sigma-q)=0, \\[2mm]
Y^{\prime}+X(2X+2\cos\sigma-q)=0,
\end{array}\right.
\end{equation}
together with 
\begin{equation}
\theta^{\prime}+\frac{x^\prime}{2y}=\cos\sigma.
\end{equation}
It should be remarked that the system \eqref{KM} is nothing but 
the K{\"a}hler magnetic curves in $\mathbb{H}^{2}(-4)$ with strength 
$\bar{q}=q-2\cos\sigma$.

\begin{example}[Reeb flows]{\rm 
According to item a) of 
Theorem \ref{magnSasaki}, Reeb flows are magnetic curves.
Choose $\theta=0$ or $\pi$ in the magnetic equations, we have 
$x(s)=\mathrm{constant}$ and
\linebreak $y(s)=\mathrm{constant}$. The coordinate $\theta$ is 
determined by $\theta'=\pm1$. 
Hence $\theta$ is an affine function of $s$.}
\end{example}

\begin{example}[Legendre $\varphi$-curves]{\rm 
According to item c) of Theorem \ref{magnSasaki}, 
Legendre $\varphi$-curves with $\kappa_1=|q|$ and 
$\kappa_2=1$ are magnetic curves. The magnetic curve 
$\gamma(s)$ is a horizontal lift 
of a Riemannian circle 
$\beta(s)=(x(s),y(s))$ with $|\kappa_\beta|=|q|$. The third 
coordinate $\theta(s)$ is determined by the 
horizontal lift condition (Legendre condition):
\[
\theta^{\prime}(s)=-\frac{x^{\prime}(s)}{2y(s)}
\]
under the prescribed initial condition.

To look for periodic trajectories, we 
restrict our attention to horizontal lifts of closed Riemannian circles.
See also \cite[Example 5.5]{Kajigaya}.

For $|\kappa_\beta|>2$, $\beta(s)$ is a closed circle 
parametrized as (see Appendix \ref{sectionA2}):
\[
(x(s),y(s))=
\left(
r\sin \mu(s)+x_0, r\left(
\frac{|q|}{2}-\cos\mu(s)
\right)
\right),
\]
where $r$ is a positive constant and $\mu(s)$ is a solution 
of the following ODE
\[
\mu^{\prime}(s)=
|q|-2\cos\mu(s).
\]
Under the initial condition $\mu(0)=0$, the solution $\mu(s)$ is given 
explicitly by
\[
\tan\frac{\mu(s)}{2}=\sqrt{\frac{|q|-2}{|q|+2}}\tan\frac{\sqrt{q^2-4}\>s}{2},
\] 
which implies
$$
\sin \mu(s)=
\frac{\sqrt{q^2-4}\sin(\sqrt{q^2-4}\>s)}{|q|+2\cos(\sqrt{q^2-4}\>s)},
\quad
\cos \mu(s)=
\frac{2+|q|\cos(\sqrt{q^2-4}\>s)}{|q|+2\cos(\sqrt{q^2-4}\>s)}.
$$
Thus $\beta(s)$ has the fundamental period $\mathsf{T}=2\pi/\sqrt{q^2-4}$.
The $\theta$-coordinate is given by 
\[
\theta(s)=\frac{1}{2} \mu(s)-\frac{|q|}{2}s,
\]
under the initial condition $\theta(0)=0$.

The horizontal lift is closed if and only if 
there exists a positive integer $m$ such that
\[
\theta\left(
s+\frac{2m\pi}{\sqrt{q^2-4}}\right)
\equiv \theta(s) \ \mod 2\pi.
\] 
Hence, the periodicity condition is equivalent to
\[
|q|=\frac{2}{\sqrt{1-(m/k)^2}}
\]
for some relatively prime positive integers $m$ and
$k$ satisfying $m/k<1$. This is precisely the criterion
found by Kajigaya in \cite{Kajigaya}.
Thus there exist countably many closed Legendre magnetic curves in $\mathrm{SL}_2\mathbb{R}$.
}
\end{example}

In the following we draw some pictures for a better understanding of
periodic Legendre magnetic curves in $\mathrm{SL}_2\mathbb{R}$.

From the previous computations we have
$$
\mu(s)=2\arctan \left(\sqrt{\frac{|q|-2}{|q|+2}}\tan\frac{\sqrt{q^2-4}\>s}{2}\right)+2h\pi,
$$
if $s\in\left(-\frac{\mathsf{T}}{2}\>,\>\frac{\mathsf{T}}{2}\right)+h\mathsf{T}$, where $h\in\mathbb{Z}$.

Fix the integers $m$ and $k$ as in the periodicity condition. We are looking now for a positive integer
$h$ such that 
$
\theta\left(\frac{\mathsf{T}}{2}+h\mathsf{T}\right)\equiv\theta\left(-\frac{\mathsf{T}}{2}\right)\quad (\text{mod }2\pi).
$
This means that $\gamma$ has $(h+1)$ "branches" to be periodic. The condition is equivalent to
$(h+1)\left(1-\frac{k}{m}\right)$ is an even number.

In the following we give some examples and draw the corresponding pictures on $\mathrm{SL}_2\mathbb{R}$
thought as a solid torus $\mathbb{S}^1\times\mathbb{D}^2$. The pictures are drawn up to a homothetic
deformation of the circle $\mathbb{S}^1$. Here $\mathbb{D}^2$ is obtained from the Poincar\'{e} half
plane $\mathbb{H}^2$ via the Cayley transformation
$$
f:\mathbb{H}^2\rightarrow\mathbb{D}^2,~
f(z)=\frac{z-i}{z+i}\>,\ \text{where}\ 
z\in\mathbb{C},~
\Im{m}(z)>0.
$$

Every figure in the next three examples is composed by four images: 
\begin{itemize}
\item the first one represents the curve $\beta$ represented in the upper half-plane;
\item the second one represents the same curve $\beta$ in the unit disc $\mathbb{D}^2$;
\item the last two pictures represent the same curve $\gamma$ on the solid torus 
$\mathbb{S}^1\times\mathbb{D}^2$ from different viewpoints.
\end{itemize}

\begin{figure}[tbh]
	\includegraphics[height=30mm]{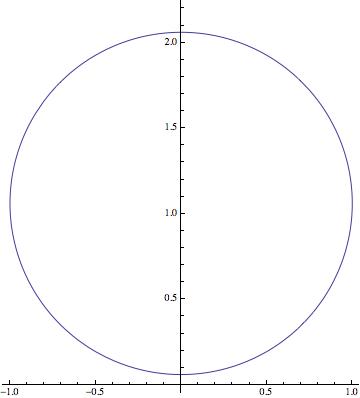}
	\quad
	\includegraphics[height=30mm]{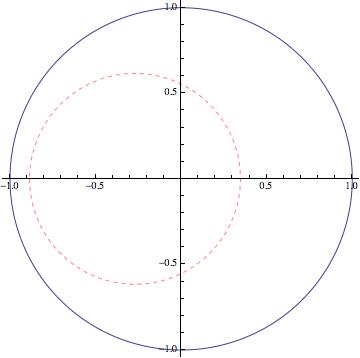}
	\quad
	\includegraphics[height=40mm]{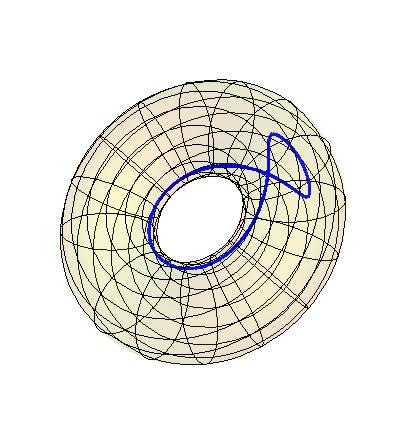}
	\quad
	\includegraphics[height=40mm]{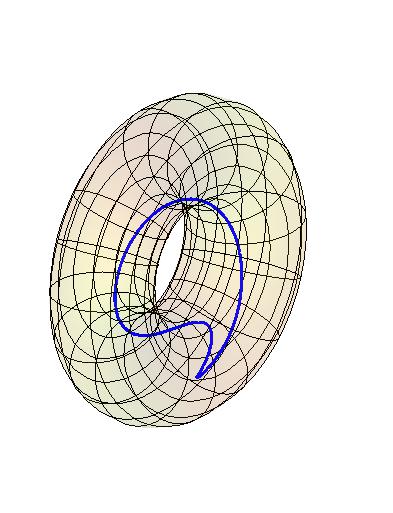}	
	\caption{ 
		$m=1$, $k=3$, $h=0$
	}
	\label{fig-L1}
\end{figure}

\begin{figure}[tbh]
	\includegraphics[height=30mm]{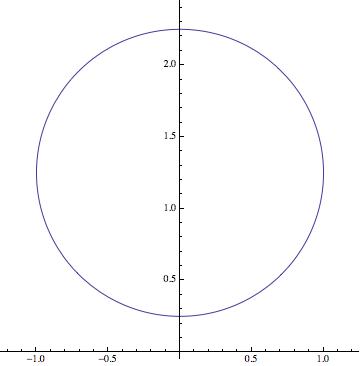}
	\quad
	\includegraphics[height=30mm]{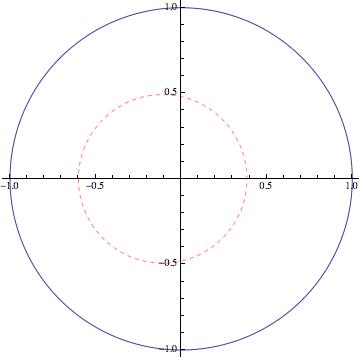}
	\quad
	\includegraphics[height=40mm]{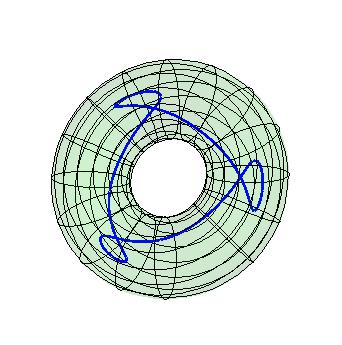}
	\quad
	\includegraphics[height=40mm]{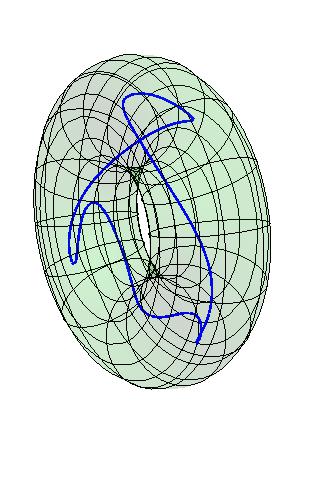}	
	\caption{ 
		$m=3$, $k=5$, $h=2$
	}
	\label{fig-L2}
\end{figure}

\begin{figure}[tbh]
	\includegraphics[height=30mm]{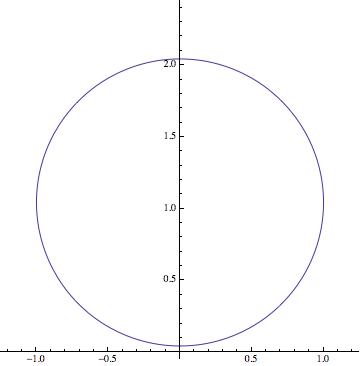}
	\quad
	\includegraphics[height=30mm]{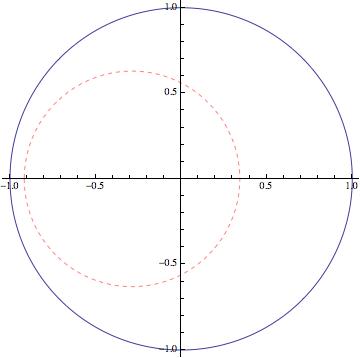}
	\quad
	\includegraphics[height=42mm]{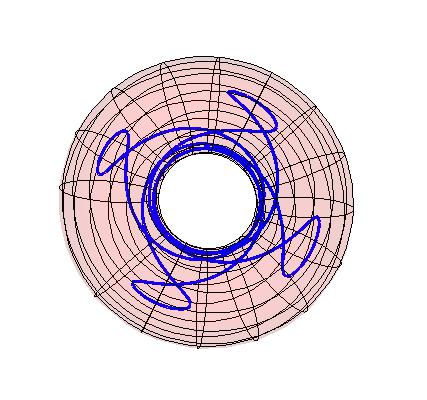}
	\quad
	\includegraphics[height=42mm]{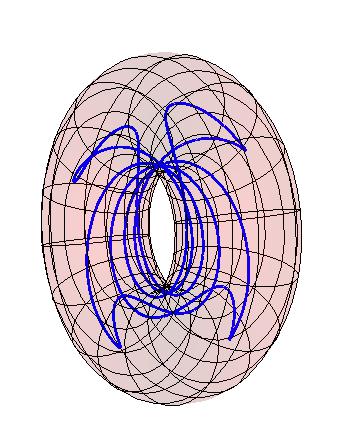}	
	\caption{ 
		$m=2$, $k=7$, $h=3$
	}
	\label{fig-L3}
\end{figure}

\FloatBarrier

\medskip


Our main interest is to classify periodic trajectories. Therefore, we consider 
magnetic curves whose projection images are closed circles.
Nevertheless, we are also interested in other kinds of magnetic trajectories
like contact magnetic trajectories over horocycles.
This study will be done in the next section. 

In the following we study periodicity of contact magnetic curves which are 
neither Reeb, nor Legendre. With this aim in view, we need to solve the system
\[
X^{\prime}-Y(2X-\bar{q})=0,
\ \
Y^{\prime}+X(2X-\bar{q})=0,
\ \
\theta^{\prime}+X=\cos\sigma,
\]
where we put $\bar{q}:=q-2\cos\sigma$.

Since $X^2+Y^2=\sin^2\sigma$, we represent $X$ and $Y$ as
\[
X=\sin\sigma\>\cos U,\ \ 
Y=\sin\sigma\>\sin U,
\]
for a certain function $U$.
Then we have
\[
X^{\prime}=-(\sin\sigma\sin U)\>U^{\prime},
\ \ 
Y^{\prime}=(\sin\sigma\cos U)\>U^{\prime},
\]
The first equation of the system is
\[
-(\sin\sigma\sin U)\>U^{\prime}-\sin\sigma\sin U
(2\sin\sigma\cos U-\bar{q})=0
\]
As $\gamma$ is neither Reeb, nor Legendre, we assume that $\sin\sigma\not=0$ and 
$\cos\sigma\not=0$, so 
\[
\{U^{\prime}+
(2\sin\sigma\cos U-\bar{q})\}\sin U=0
\]
The second equation becomes
\[
\{U^{\prime}+
2\sin\sigma \cos U-\bar{q})\}\cos U=0.
\]
Hence we obtain
\begin{equation}
\label{ODE:U}
U^{\prime}+2\sin\sigma \cos U-\bar{q}=0.
\end{equation}

This ODE can be solved directly. With the new variable $t=\tan (U/2)$, 
the equation~\eqref{ODE:U} can be rewritten as
\[
2\frac{dt}{ds}
=(\bar{q}+2\sin\sigma)t^2+(\bar{q}-2\sin\sigma).
\]

We have to distinguish several cases:

{\bf Case 1.} $\bar{q}+2\sin\sigma=0$

This is equivalent to $q=2\sqrt{2}\sin(\sigma-\pi/4)$. 
Under the initial condition $U(0)=0$, we get $t(s)=-2s\sin\sigma$. Thus we obtain
\[
U(s)=-2\arctan\left(2s\sin\sigma\right).
\]

{\bf Case 2.} $\bar{q}-2\sin\sigma=0$

In this case, we have $q=2\sqrt{2}\sin(\sigma+\pi/4)$ and 
$\mathrm{d}t=(2\sin\sigma)t^2\>\mathrm{d}s$.
The solution of this ODE with initial condition $U(0)=\pi/2$ is
$t(s)=\frac{1}{1-2s\sin\sigma}$.
Thus 
\[
U(s)=2\arctan\frac{1}{1-2s\sin\sigma}~.
\]

{\bf Case 3.} $\bar{q}^2-4\sin^2\sigma>0$

We need to solve the ODE:
$\displaystyle
\frac{dt}{ds}=
\frac{\bar{q}+2\sin\sigma}{2}
\left(
t^2+
\frac{\bar{q}-2\sin\sigma}
{\bar{q}+2\sin\sigma}
\right),
$
where 
$\frac{\bar{q}-2\sin\sigma}{\bar{q}+2\sin\sigma}>0$.
Solving this equation with the initial condition $U(0)=0$, 
we obtain
\[
U(s)=2\arctan
\left(
\sqrt{\frac{\bar{q}-2\sin\sigma}
{\bar{q}+2\sin\sigma}}
\tan
\frac{s\>\sqrt{\bar{q}^2-4\sin^2\sigma}}{2}
\right).
\]

{\bf Case 4.} $\bar{q}^2-4\sin^2\sigma<0$

We need to solve the ODE:
$\displaystyle
\frac{dt}{ds}=
\frac{2\sin\sigma+\bar{q}}{2}
\left(
t^2-
\frac{2\sin\sigma-\bar{q}}
{2\sin\sigma+\bar{q}}
\right),$
where 
$\frac{2\sin\sigma-\bar{q}}{2\sin\sigma+\bar{q}}>0$.
Setting the initial condition $U(0)=0$, we find
\[
U(s)=-2\arctan
\left(
\sqrt{
\frac{2\sin\sigma+\bar{q}}{2\sin\sigma-\bar{q}}
}
\tanh\frac{s\>\sqrt{4\sin^2\sigma-\bar{q}^2}}{2}
\right).
\]

\medskip

To look for closed trajectories, we need to demand that 
$|\bar{q}|=|q-2\cos\sigma|>2$. This condition immediately implies that cases 1, 2
and 4 cannot occur.

Let $\gamma(s)=(x(s),y(s),\theta(s)$ be a 
periodic contact magnetic curve which is neither Reeb nor 
Legendre. 
Let us denote by $\mathsf{T}$ the fundamental period of a 
periodic contact magnetic curve 
$\gamma(s)=(x(s),y(s),\theta(s))$.
Namely 
\[
x(s+\mathsf{T})=x(s), \ \ 
y(s+\mathsf{T})=y(s)\quad
\text{and} \quad
\theta(s+\mathsf{T})\equiv\theta(s) \mod 2\pi.
\]
The projected curve of 
$\beta(s)=(x(s),y(s))$ is 
is a closed Riemannian circle determined by
\[
\frac{x^{\prime}(s)}{2y(s)}=\sin\sigma\>\cos U(s),
\ \ 
\frac{y^{\prime}(s)}{2y(s)}=\sin\sigma\>\sin U(s).
\]
The second equation implies that 
\[
\frac{\mathrm{d}}{\mathrm{d}s}\log y(s)=2\sin\sigma\>\sin 
U(s).
\]
Since 
\[
\frac{\mathrm{d}U}{\mathrm{d}s}=-2\sin\sigma \>\cos U+\bar{q},
\]
we have
\[
\log y(s)=
\int 
\frac{2\sin\sigma \sin U}
{\bar{q}-2\sin\sigma\>\cos U}
\>\mathrm{d}U
=\log|\bar{q}-2\sin\sigma\>\cos U(s)|+\mathrm{constant}.
\] 
Thus we obtain
\[
y(s)=\bar{r}\big(\bar{q}-2\sin\sigma\>\cos U(s)\big)
\]
for some $\bar{r}\not=0$.
From this we have
\[
x^{\prime}(s)=
(2\sin\sigma)\cos U(s)\cdot
(\bar{r}U^{\prime}(s)),
\]
which leads to
\[
x(s)=\int(2\bar{r}\sin\sigma)\cos U\>\mathrm{d}U
=(2\bar{r}\sin\sigma)\sin U+x_0.
\] 
We notice that $\bar{s}=(\sin\sigma)s$ is 
the arclength parameter of $\beta$. 
The curvature of $\beta$ is $\kappa_\beta=\frac{\bar{q}}{\sin\sigma}$.

The $\theta$-coordinate is 
determined by 
\[
\theta^{\prime}(s)=\cos\sigma-\frac{x^{\prime}(s)}{2y(s)}=
\cos\sigma-\sin\sigma\cos U(s).
\]
Since $U'(s)=-(2\sin\sigma)\cos U(s)+\bar{q}$, 
we get
\[
\theta'(s)=\cos\sigma-\frac{\bar{q}}{2}+\frac{U'(s)}{2}.
\]
Thus, the solution satisfying $\theta(0)=\theta_0$, is given by 
\[
\theta(s)=
\left(
\cos\sigma-\frac{\bar{q}}{2}
\right)
s
+\frac{U(s)}{2}+\theta_0.
\]

The periodicity of $x(s)$ and $y(s)$ implies that
\[\begin{array}{l}
x(s+\mathsf{T})-x(s)=
\ 2\bar{r}\sin\sigma\>\big[(\sin U(s+\mathsf{T})-\sin U(s)\big],\\[2mm]
y(s+\mathsf{T})-y(s)=
-2\bar{r}\sin\sigma\>\big[\cos U(s+\mathsf{T})-\cos U(s)\big],
\end{array}
\]
for all $s$. These formulas yield 
$$\sin U(s+\mathsf{T})=\sin U(s)
\quad \text{and}\quad
\cos U(s+\mathsf{T})=\cos U(s),\quad \text{for all } s.
$$
Thus we obtain
\[
U(s+\mathsf{T})\equiv U(s)
\mod 2\pi.
\]
Under the hypothesis $U(0)=0$, we have 
\begin{equation}\label{eq:per4}
U(\mathsf{T})=2k\pi
\end{equation}
for some integer $k$.
Since 
\begin{equation}\label{eq:per5}
\tan\frac{U(s)}{2}=\sqrt{\frac{\bar{q}-2\sin\sigma}{\bar{q}+2\sin\sigma}}
\tan \left(
\frac{\sqrt{\bar{q}^2-4\sin^2\sigma}}{2}\>s
\right),
\end{equation}
we have 
\[
\tan\left(\frac{\>\mathsf{T}\sqrt{\bar{q}^2-4\sin^2\sigma}}{2}
\right)=0.
\]
This is equivalent to 
\[
\frac{\>\mathsf{T}\sqrt{\bar{q}^2-4\sin^2\sigma}}{2}=m\pi
\]
for some integer $m$.

The periodicity of $\theta(s)$ implies 
\begin{equation}\label{eq:per3}
\theta(s+\mathsf{T})-\theta(s)
=\left(
\cos\sigma-\frac{\bar{q}}{2}
\right)\>\mathsf{T}
+\frac{U(s+\mathsf{T})-U(s)}{2}\equiv 0 \mod 2\pi.
\end{equation}

From \eqref{eq:per4} and \eqref{eq:per3}, 
we get
\[
\left(\cos\sigma-\frac{\bar{q}}{2}\right)
\mathsf{T}+k\pi\equiv 0 \mod 2\pi.
\]
which leads
\begin{equation}\label{eq:per6}
\mathsf{T}=\frac{2k\pi}{\bar{q}-2\sin\sigma},
\end{equation}
(possible for other integer $k$).
From \eqref{eq:per5} and \eqref{eq:per6}, 
we find
\[
\sqrt{\bar{q}^2-4\sin^2\sigma}=\frac{m}{k}(\bar{q}-2\cos\sigma).
\]
Solving this equation we have 
\begin{equation}
\label{eq:q}
q=\frac{2a\cos\sigma\pm\sqrt{2(1-a\cos(2\sigma))}}{\frac{1+a}{2}},
\end{equation}
where 
\begin{equation}
a=1-2\left(
\frac{m}{k}
\right)^2.
\label{eq:a}
\end{equation}

\medskip

We now state the following result.

\medskip

\begin{theorem}
The set of all periodic magnetic curves on the special linear group
$\mathrm{SL}_2\mathbb{R}$ can be quantized in the set of rational numbers.
\end{theorem}
\proof The proof is a consequence of equations \eqref{eq:q} and \eqref{eq:a}.
\endproof

\begin{remark}{\rm 
When $\cos\sigma=0$, the strength $q$ has the form
\[
q=\frac{\pm 2}{\sqrt{1-(m/k)^2}}.
\]
This expression coincides with the Kajigaya's criterion.}
\end{remark}

\medskip

In the following we draw some pictures of periodic non-Reeb and 
non-Legendre magnetic curves in $\mathrm{SL}_2\mathbb{R}$.
Every figure in the next three examples is composed by four images,
keeping the same convention as before.

\begin{figure}[tbh]
	\includegraphics[height=30mm]{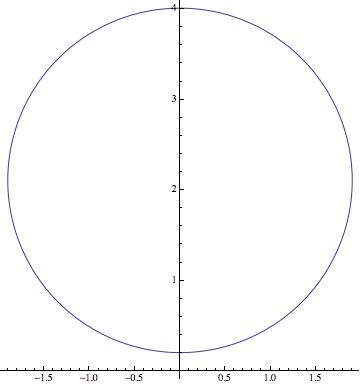}
	\quad
	\includegraphics[height=30mm]{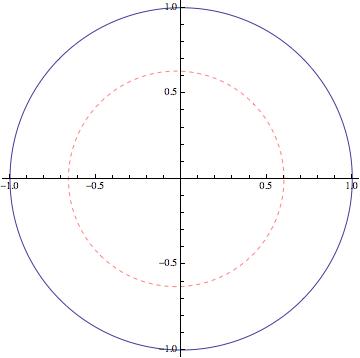}
	\quad
	\includegraphics[height=40mm]{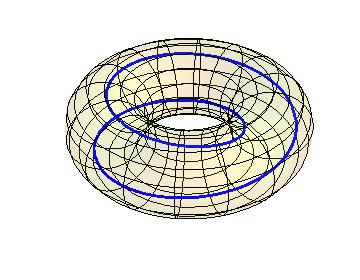}
	\quad
	\includegraphics[height=40mm]{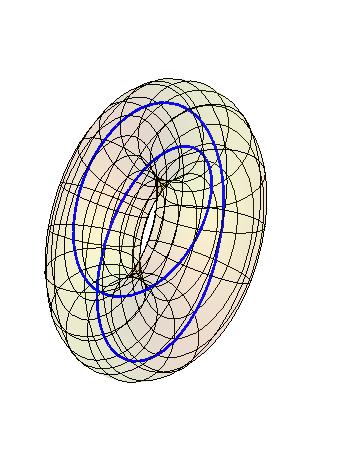}	
	\caption{ 
		$m=1$, $k=3$, $\sigma=\frac{2\pi}{5}$
	}
	\label{fig-L1}
\end{figure}

\begin{figure}[tbh]
	\includegraphics[height=30mm]{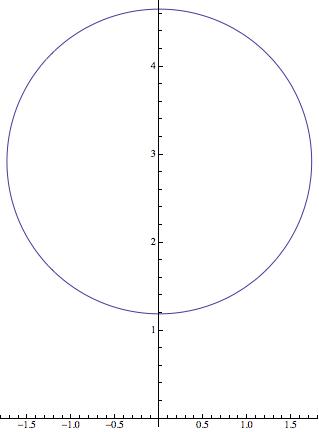}
	\quad
	\includegraphics[height=30mm]{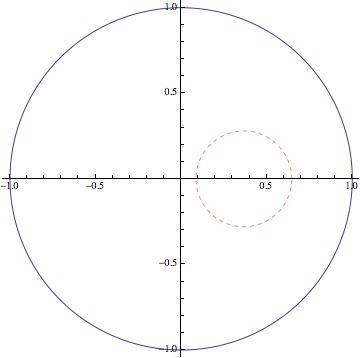}
	\quad
	\includegraphics[height=40mm]{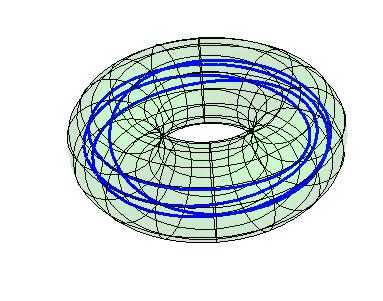}
	\quad
	\includegraphics[height=40mm]{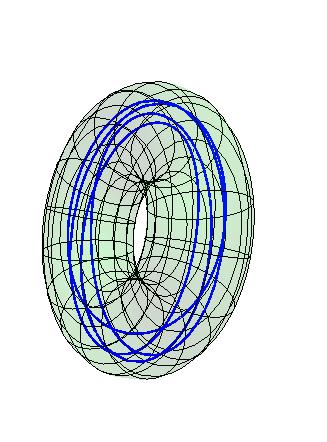}	
	\caption{ 
		$m=3$, $k=5$, $\sigma=\frac{\pi}{3}$
	}
	\label{fig-L2}
\end{figure}

\FloatBarrier

\begin{remark}{\rm 
As we have mentioned in Remark \ref{Remark1}, 
we may consider the one-parameter family of 
homogenous metrics $g_{\nu}$.
In \cite{DIMN2}, we have studied contact 
magnetic curves in cosymplectic manifolds. 
In particular, Nistor investigated 
contact magnetic curves in \cite{Nistor}. 
According to \cite{Nistor}, contact magnetic curves in 
$\mathbb{H}^{2}(-4)\times\mathbb{R}$ are classified as 
follows:
\begin{itemize}
\item  a geodesic line 
$(x_0,y_0,t_0\pm{s})$ through a point 
$(x_0,y_0,t_0)$.

\item 
a horocycle $\beta_0\times\{t_0\}$ 
in every point $t_0\in\mathbb{R}$, 
where $\beta_0$ denotes a (open) circle tangent
to the ideal boundary or a horizontal line in 
$\mathbb{H}^2(-4)$, of constant curvature 
$\kappa^2=4$;
\item  a non-degenerate cylindrical helix on 
$\beta\times\mathbb{R}$, 
where $\beta$ denotes both a Euclidean and
a hyperbolic circle in $\mathbb{H}^2(-4)$.
\end{itemize}
One can unify 
the results in this paper and those in \cite{Nistor}.
}
\end{remark}

\section{One-parameter subgroups} 
Let $G$ be a Lie group equipped with a 
left invariant Riemannian metric and let $\mathfrak{g}$ be its Lie algebra. 
Then an arclength parametrized curve $\gamma$ 
in $G$ is said to be \emph{homogenous} if 
there exists a one-parameter subgroup $\{\exp(tX)\}$ such that 
$\gamma(t)$ is expressed as $\gamma(t)=a\>\exp(tX)$ for some 
$a\in G$ and unit vector $X\in\mathfrak{g}$. 
In this section we investigate homogeneous magnetic trajectories.
 
\subsection{Homogeneous geodesics}
In Lie groups with bi-invariant Riemannian metric,
all the geodesics starting at the identity are
one parameter subgroups.
However, if the metric is only left invariant,
one-parameter subgroups are not necessarily
geodesics. Here we study geodesic in 
$\mathrm{SL}_2\mathbb{R}$ which are one-parameter subgroups 
of $\mathrm{SL}_2\mathbb{R}$.

\begin{proposition}
The one-parameter subgroup $\exp(tX)$ of an element
\[
X=\left(
\begin{array}{cc}
X_{11} & X_{12}\\
X_{21} & -X_{11}
\end{array}
\right)
\]
of $\mathfrak{sl}_{2}\mathbb{R}$ is 
given explicitly  by as follows{\rm:}
\begin{itemize}
\item If $\det X=0$, then 
\[
\exp(tX)=\left(
\begin{array}{cc}
1+tX_{11} & tX_{12}\\
tX_{21}& 1-tX_{11}
\end{array}
\right).
\]
Every $\exp(tX)$ induces a parabolic transformation on 
$\mathbb{H}^2(-4)$.  
\item $\det X=\delta^2>0$, then 
\[
\exp(tX)=\left(
\begin{array}{cc}
\cos(\delta t)+\frac{X_{11}}{\delta}\sin(\delta t)
 & \frac{X_{12}}{\delta}\sin(\delta t)
 \\
 \frac{X_{21}}{\delta}\sin(\delta t)
 & 
\cos(\delta t)-\frac{X_{11}}{\delta}\sin(\delta t)
\end{array}
\right).
\]
Every $\exp(tX)$ induces an elliptic transformation on 
$\mathbb{H}^2(-4)$.  
\item $\det X=-\delta^2<0$
\[
\exp(tX)=\left(
\begin{array}{cc}
\cosh(\delta t)+\frac{X_{11}}{\delta}\sinh(\delta t)
 & \frac{X_{12}}{\delta}\sinh(\delta t)
 \\
 \frac{X_{21}}{\delta}\sinh(\delta t)
 & 
\cosh(\delta t)-\frac{X_{11}}{\delta}\sinh(\delta t)
\end{array}
\right).
\]
Every $\exp(tX)$ induces a hyperbolic transformation on 
$\mathbb{H}^2(-4)$.  
\end{itemize}
\end{proposition}

Take an element $X\in
\mathfrak{sl}_{2}\mathbb{R}$, then 
the acceleration vector field $\nabla_{\gamma^\prime}\gamma^{\prime}$ of $\gamma(t)=\exp(tX)$
at the origin is ${\mathsf{U}}(X,X)$ because of \eqref{LC}. 
Thus we obtain the following well known criterion:

\begin{proposition}
A one parameter subgroup
$\{ \exp (t X) \}_{t \in \mathbb{R}},\ X \in 
\mathfrak{sl}_{2}\mathbb{R}$
is a geodesic if and only if
${\mathsf{U}}(X,X)=0$. 
\end{proposition}

Now we apply this criterion for 
$X=aE_{1}+bE_{2}+cE_{3} \in \mathfrak{sl}_2\mathbb{R}$. 
By using \eqref{U-tensor}, ${\mathsf{U}}(X,X)$ is computed as
\[
{\mathsf{U}}(X,X)=2c(b-a)E_{1}+2c(b-a)E_{2}+2(a^{2}-b^{2})E_{3}.
\]
Thus we obtain 
\begin{corollary}
A one-parameter subgroup $\{\exp(tX)\}$ of 
$X=aE_1+bE_2+cE_3$ is a geodesic in 
$\mathrm{SL}_2\mathbb{R}$ if and only if either 
\begin{itemize}
\item $a=b$ or 
\item $c=0$ and $a=-b$.
\end{itemize}
\end{corollary}
In particular 
$\exp(t E_{3})$ is the
only geodesic
among
$\exp(t E_{i})$, ($i=1,2,3$).
Here we describe the spurs
of $\exp(tE_{i})$.
Direct computations show the following formula:
\[
\exp (tE_1)=
\begin{pmatrix} 1 & \sqrt{2}t \\ 0 & 1
\end{pmatrix}.
\]
Thus the coordinate expression of $\exp(tE_1)$ is 
\[
x(t)=\sqrt{2}t,\ \ y=1,\ \ \theta(t)=0.
\]
Hence the spur of  
$\exp(tE_1)$ in the universal covering 
$\widetilde{\mathrm{SL}_2\mathbb{R}}=\mathbb{H}^2(-4)
\times\mathbb{R}$
is the line through $(0,1,0)$ 
parallel to the $x$-axis.
The contact angle of $\exp(tE_1)$ is $\pi/4$.

Note that for all $t \in \mathbb{R}$,
\begin{equation}
\label{Exp-E1}
\begin{pmatrix} 1 & x \\ 0 & 1
\end{pmatrix}=
\exp \Big(\frac{x}{\sqrt{2}}E_1\Big).
\end{equation}
Hence the mapping
\[
\exp\Big(\frac{{\scriptstyle \bullet}}{\sqrt{2}}E_1\Big):
(\mathbb{R}(x),+) \longrightarrow N
\]
is a Lie group isomorphism.

Next the trace of $\exp(tE_2)$ is given by
\begin{equation}
\label{Exp-E2}
\exp(tE_2)=
\begin{pmatrix}
1 & 0 \\ \sqrt{2}t & 1
\end{pmatrix}.
\end{equation}
This curve has the parametrization
\[
x(t)=\frac
{\sqrt{2}t}{1+2t^2}, \ \ 
y(t)=\frac{1}{1+2t^2},
\ \ 
\theta(t)=\arctan(-\sqrt{2}t)
\]
and contact angle $3\pi/4$.
The projected curve $(x(t),y(t))$ of 
$\exp(tE_2)$ in $\mathbb{H}^2(-4)$ is the horocycle
\[
x^2+\left(y-\frac{1}{2}\right)^2=\frac{1}{4}.
\]
These two one-parameter subgroups 
$\{\exp(tE_1)\}$ and $\{\exp(tE_2)\}$ are 
\textit{not} geodesics. However as we will see later, they are 
contact magnetic curves.

The one-parameter subgroup
\[
\exp(t E_3)=
\begin{pmatrix}
e^t & 0 \\ 0 & e^{-t}
\end{pmatrix}
\]
is a Legendre geodesic with parametrization
\[
x(t)=0,\ \ y(t)=e^t,\ \ \theta(t)=0.
\]
The spur 
of $\exp(tE_3)$ is the 
$y$-axis in the universal covering 
$\mathbb{H}^2(-4)\times\mathbb{R}$.
The projected curve $(x(t),y(t))$ 
is a vertical line, \textit{i.e.}, 
a horocycle with base point 
$\infty$. 

\begin{remark}{\rm
A homogeneous Riemannian manifold 
$M=G/K$ is called a 
\emph{space with homogeneous geodesics} or a 
\emph{Riemannian g.o.~space} if every geodesic 
$\gamma(t)$ of $M$ is an orbit of a one-parameter subgroup of $G$ \cite{KV}. 
Naturally reductive homogenous spaces are typical examples 
of Riemannian g.o. spaces. (For more informations, we refer to \cite{Av}).
}

\end{remark}

Explicit parametrization of geodesics 
in $\mathbb{H}^{2}(-1)\times\mathbb{R}$ have been obtained in 
\cite{Sitzia} (see also \cite{Profir}). 
Note that in \cite{Profir,Sitzia}, 
$\widetilde{\mathrm{SL}_2\mathbb{R}}$ is realized as
the product manifold of unit disk and the real line.

\subsection{Homogeneous magnetic trajectories}

Let us consider magnetic equation for one-parameter 
subgroups. 
First we prepare the following proposition.
\begin{proposition}
The endomorphism field $\varphi$ satisfies
\[
\varphi (aE_1+bE_2+cE_3)=
-\frac{c}{\sqrt{2}}(E_1+E_2)+\frac{1}{\sqrt{2}}(a+b)
E_3
\]
for any $X=aE_1+bE_2+cE_3\in\mathfrak{sl}_{2}
\mathbb{R}$.
\end{proposition}
From this proposition, the magnetic equation 
$\nabla_{\gamma^{\prime}}\gamma^{\prime}=q\varphi\gamma^{\prime}$ for 
$\gamma(t)=\exp(tX)$ may be rewritten as
\[
2c(b-a)=-\frac{cq}{\sqrt{2}},\ \ 
2(a^2-b^2)=\frac{q}{\sqrt{2}}(a+b).
\]
\begin{proposition}
The one parameter subgroup  
$\exp(tX)$ of $X=aE_1+bE_2+cE_3
\in\mathfrak{sl}_{2}\mathbb{R}$ is 
a magnetic curve with strength $q\not=0$ 
if and only if $a-b=q/(2\sqrt{2})$.
\end{proposition}

\begin{remark}
{\rm
If $b=-a$, in order to 
have a non-geodesic 
magnetic curve with strength $q$, we need to ask that 
$c\not=0$. In this situation, $a=q/(4\sqrt{2})$.}
\end{remark}

Since the Reeb vector field is 
\[
\xi=\frac{1}{\sqrt{2}}(E_1-E_2),
\]
the contact angle $\sigma$ of the one parameter 
subgroup $\exp(tX)$ of 
$X=aE_1+bE_2+cE_3\in\mathfrak{sl}_{2}\mathbb{R}$ 
defined by 
\[
\cos\sigma=\frac{a-b}{\sqrt{2}\sqrt{a^2+b^2+c^2}}.
\]
Therefore, if $\exp(tX)$ is a Legendre curve, it 
necessarily should be a geodesic.

\begin{example}
{\rm 
Let $X=aE_1+bE_2+cE_3\in\mathfrak{sl}_{2}\mathbb{R}$ such that 
$\det X=0$.
We shall emphasize two 
situations: 
(\texttt{(i)} $b=-a$ and $c\not=0$ and 
\texttt{(ii)} $c=0$. 
In the first case \texttt{(i)}, $c=\pm\sqrt{2}a$ and 
\[
\exp(tX)=
\left(
\begin{array}{cc}
1\pm\sqrt{2}at & \sqrt{2}at\\
-\sqrt{2}at & 1\mp\sqrt{2}at 
\end{array}
\right).
\]
The Iwasawa decomposition of $\exp(tX)$ is 
given by
\[
x(t)=\frac{\mp 4a^2t^2}{2a^2t^2+(1\mp \sqrt{2}at)^2},
\ \
y(t)=\frac{1}{2a^2t^2+(1\mp \sqrt{2}at)^2},
\ \ 
e^{i\theta(t)}=
\frac{1\mp\sqrt{2}at+i\sqrt{2}at}{\sqrt{2a^2t^2+(1\mp \sqrt{2}at)^2}}.
\]
The contact angle $\sigma$ satisfies 
$\cos\sigma=\mathrm{sgn}(a)/\sqrt{2}$, 
thus
$\sigma=\pi/4$ or $3\pi/4$.
The projected curve is the horocycle
$(x\pm 1)^2+(y-1)^2=1$.
The curve $\exp(tX)$ is a 
contact magnetic curve with 
strength $q$ if and only if 
$a=q/(4\sqrt{2})$. 

In the second case \texttt{(ii)}, 
we consider $Y=aE_1+bE_2\not=0$. Then 
$\det Y=0$ if and only if $a=0$ or $b=0$, namely 
$Y=aE_1$ or $Y=bE_2$.

$\mathtt{(ii.1)}$
In the case $Y=aE_1$, from \eqref{Exp-E1}, we obtain that
$\exp(taE_1)$ is parametrized as 
$(x(t),y(t),\theta(t))=(a\sqrt{2}t,1,0)$.
It is a contact magnetic curve 
with strength $q$ if and only if 
$a=q/(2\sqrt{2})$.

The trajectory lies in the 
nilpotent subgroup $N$. 
The contact angle is
$\cos\sigma=\pm 1/\sqrt{2}$. Thus $\sigma=\pi/4$ or 
$3\pi/4$.
The projected curve is the horizontal line $y=1$ (horocycle with base 
point $\infty$).

$\mathtt{(ii.2)}$ 
In the case $Y=bE_2$, using \eqref{Exp-E2}, it follows that
$\exp(tbE_2)$ is parametrized as
\[
x=\frac{\sqrt{2}bt}{2b^2t^2+1},\ \ 
y=\frac{1}{2b^2t^2+1},\ \ 
e^{i\theta}=\frac{1-i\sqrt{2}bt}{\sqrt{2b^2t^2+1}}.
\]
From these relations we obtain
$x^2+\left(y-\frac{1}{2}\right)^2=\frac{1}{4}$.
Thus $(x(t),y(t))$ is a horocycle. 
Note that $\exp(tY)$ has contact angle $\pi/4$ or 
$3\pi/4$. It is a 
contact magnetic curve with strength $q$ 
if and only if $b=-q/(2\sqrt{2})$.
}
\end{example}

\begin{remark}
{\rm 
We can write the projection curve 
$\beta$ for the general situation 
$\exp(tX)$ with $X=aE_1+bE_2+cE_3$ 
with $\det X=-(c^2+2ab)=0$:
\[
x(t)= \frac{\sqrt{2}\{(a+b)t+c(b-a)t^2\}}
{2b^2t^2+(1-tc)^2},\ \
y(t)= \frac{1}{2b^2t^2+(1-tc)^2}.
\]
If the parameter $t$ us eliminated, we obtain 
the following equation for $\beta$:
\begin{enumerate}
\item If $b=0$ (and hence also $c=0$), then $y-1=0$.
\item If $b\not=0$, then
\[
\left(x-\frac{c}{b\sqrt{2}}\right)^2
+\left(y-\frac{b-a}{2b}\right)^2=\frac{(b-a)^2}{4b^2}.
\]
\end{enumerate}
These formulas actually show that 
$\beta$ is a horocycle.  

Note that $t$ is not, in general, 
the arc length parameter for $\exp(tX)$.
In case $t$ is the arc length 
parameter for $\exp(tX)$, 
then $a^2+b^2+c^2=1$. The condition 
$\det X=0$ implies $|a-b|=1$ and thus, 
the contact angle $\sigma$ is obtained as 
$\cos\sigma=\pm 1/\sqrt{2}$
meaning that $\sigma=\pi/4$ or $3\pi/4$. 
The magnetic condition is $q=2\mathrm{sgn}(a-b)\>\sqrt{2}$.
Computing $\bar{q}=q-2\cos\sigma=\sqrt{2}\mathrm{sgn}(a-b)$.
Hence the necessarily condition for periodicity, 
namely $|\bar{q}|>2$ fails. The curvature 
$\kappa_{\beta}=\bar{q}/\sin\sigma=2\>\mathrm{sgn}(a-b)$.
}
\end{remark}

\begin{example}
{\rm 
Let us assume that $\det X>0$.  
First we consider $X=a(E_1-E_2)+cE_3$.
The one-parameter subgroup is 
given by 
\[
\exp(tX)=\left(
\begin{array}{cc}
\cos(\delta t)+\frac{c}{\delta}\sin(\delta t)
 & \frac{\sqrt{2}a}{\delta}\sin(\delta t)
 \\
- \frac{\sqrt{2}a}{\delta}\sin(\delta t)
 & 
\cos(\delta t)-\frac{c}{\delta}\sin(\delta t)
\end{array}
\right).
\]
Note that $\delta^2=2a^2-c^2$.
We perform the Iwasawa decomposition.
\begin{align*}
x(t)=&\frac
{-\frac{2\sqrt{2}ac}{\delta^2}\sin^2(\delta t)}
{\frac{2a^2}{\delta^2}\sin^2(\delta t)+(\cos(\delta t)-\frac{c}
{\delta}\sin(\delta t))^2},
\\
y(t)=&\frac{1}{\frac{2a^2}{\delta^2}\sin^2(\delta t)+(\cos(\delta t)-\frac{c}
{\delta}\sin(\delta t))^2},
\\
e^{i\theta(t)}=&
\frac{\cos(\delta t)-\frac{c}{\delta}\sin(\delta t)+
\frac{\sqrt{2}ai}{\delta}\sin(\delta t)}{\sqrt{\frac{2a^2}{\delta^2}
\sin^2(\delta t)+(\cos(\delta t)-\frac{c}
{\delta}\sin(\delta t))^2}}.
\end{align*}
The projected curve $\beta$ 
is expressed as
\[
\left(
x+\frac{c}{a\sqrt{2}}
\right)^2
+(y-1)^2=\frac{c^2}{2a^2}.
\]
Next we treat the case $Y=aE_1+bE_2$, for which $\det Y=-2ab=\delta^2>0$.
We have
\[
\exp(tY)=\left(
\begin{array}{cc}
\cos(\delta t)
 & \frac{\sqrt{2}a}{\delta}\sin(\delta t)
 \\
\frac{\sqrt{2}b}{\delta}\sin(\delta t)
 & 
\cos(\delta t)
\end{array}
\right).
\]
Hence the projected curve $\beta$ is 
\[
x(t)=\frac{\sqrt{2}(a+b)\sin(\delta t)
\cos(\delta t)}{\frac{2b^2}{\delta}\sin^2(\delta t)+
\delta\cos^2(\delta t)},
\ \ 
y(t)=\frac{1}{\frac{2b^2}{\delta^2}\sin^2(\delta t)+
\cos^2(\delta t)}.
\]
The implicit form of $\beta$ is 
\[
x^2
+\left(
y-\frac{b-a}{2b}
\right)^2=\left(
\frac{a+b}{2b}
\right)^2.
\]
Comparing this with \eqref{circle-eq}, we obtain
$\displaystyle
\kappa_{\beta}=\frac{2(b-a)\>\mathrm{sgn}(b)}{|b+a|}$.
The contact angle satisfies
$\displaystyle
\cos\sigma=\frac{a-b}{\sqrt{2}\>\sqrt{a^2+b^2}}$.
In both cases, the magnetic curves are periodic.}
\end{example}
\begin{example}{\rm 
Assume now that $\det X=-\delta^2<0$, for $X=aE_1+bE_2+cE_3$.
In this case we have
\[
\exp(tX)=
\left(
\begin{array}{cc}
\cosh (\delta t) +\frac{c}{\delta}\sinh(\delta t)
& \frac{\sqrt{2}a}{\delta}\sinh(\delta t)
\\
\frac{\sqrt{2}b}{\delta}\sinh(\delta t)
&
\cosh (\delta t) -\frac{c}{\delta}\sinh(\delta t)
\end{array}
\right).
\]

Let us compute the coordinates $(x,y,\theta)$ of 
$\exp(tX)$.
\begin{align*}
x(t)=&\frac{
\frac{\sqrt{2}(a+b)}{\delta}
\sinh(\delta t)\cosh(\delta t)+
\frac{\sqrt{2}(b-a)c}{\delta^2}
\sinh(\delta t)\cosh(\delta t)
}
{\frac{2b^2}{\delta^2}\sinh^{2}(\delta t)+
\left(\cosh(\delta t)-\frac{c}{\delta}\sinh(\delta t)
\right)^2},
\\
y(t)=&
\frac{1}
{\frac{2b^2}{\delta^2}\sinh^{2}(\delta t)+
\left(\cosh(\delta t)-\frac{c}{\delta}\sinh(\delta t)
\right)^2}.
\end{align*}
Then the projected curve $\beta$ is 
described as follows:
\begin{enumerate}
\item If $b\not=0$, then $\beta$ is a part of an 
open circle: 
\[
\left(x-\frac{c}{b\sqrt{2}}\right)^2+
\left(y+\frac{a-b}{2b}\right)^2=
\frac{(a+b)^2+2c^2}{4b^2}.
\]
\item If $b=0$, then $\beta$ is 
\[
\sqrt{2}\delta x+\mathrm{sgn}(c)a(1-y)=0.
\]
\end{enumerate}
In both cases, the signed curvature of $\beta$ is 
\[
\kappa_{\beta}=\frac{2\>\mathrm{sgn}(b)(b-a)}
{\sqrt{(a+b)^2+2c^2}}.
\]

}
\end{example}

\appendix
\section{Curve theory in $\mathbb{H}^2(-4)$}\label{sectionA2}
\subsection{Frenet formula in $\mathbb{H}^2(-4)$}
Let $\beta(s)=(x(s),y(s))$ be an arclength parametrized curve 
in $\mathbb{H}^2(-4)$. 
Take a global orthonormal frame field
\[
\bar{\epsilon}_1=2y\>\frac{\partial}{\partial x},
\  \
\bar{\epsilon}_2=2y\>\frac{\partial}{\partial y},
\]
then the unit tangent vector field $
\overline{T}(s)=\beta^{\prime}(s)$ is 
represented by
\[
\beta^{\prime}(s)=\frac{1}{2y(s)}\left(
x^{\prime}(s)\bar{\epsilon}_1
+y^{\prime}(s)\bar{\epsilon}_2
\right).
\]
The unit normal vector field 
$\overline{N}(s)$ is
\[
\overline{N}(s)=J\overline{T}(s)
=\frac{1}{2y(s)}\left(
-y^{\prime}(s)\bar{\epsilon}_1
+x^{\prime}(s)\bar{\epsilon}_2
\right).
\]
Then we obtain
\[
\overline{\nabla}_{\beta^{\prime}}\overline{T}=
\frac{1}{2y^2}\left\{
(x^{\prime\prime}y-2x^{\prime}y^{\prime})\bar{\epsilon}_1
+
(y^{\prime\prime}y
+(x^\prime)^2-(y^\prime)^2)\bar{\epsilon}_2
\right\}=\kappa_{\beta}J\overline{T}(s).
\]
Thus the (signed) curvature $\kappa_\beta$ is 
\[
\kappa_{\beta}=
\frac{x^{\prime}y^{\prime\prime}-
x^{\prime\prime}y^{\prime}}{4y^2}
+\frac{x^{\prime}}{y}.
\]

\subsection{Riemannian circles in $\mathbb{H}^2(-4)$}
In a Euclidean space, Riemannian circles are nothing but usual circles.
Hence, every Riemannian circle in a Euclidean space is simple and closed.
Besides, Riemannian circles in spheres are small circles and hence, 
every Riemannian circle in spheres is simple and closed, too. 
Nevertheless, in hyperbolic spaces, Riemannian circles are not necessarily closed.

\begin{proposition}
Every Riemannian circle $\beta$
of curvature $\kappa_\beta=k$ in $\mathbb{H}^2(-4)$ is a horizontal line, 
a vertical line or a part of a circle, given by the following formula:
\begin{equation}\label{circle-eq}
(x-a)^2+(y-rk)^2=4r^2.
\end{equation}
In particular, a Riemannian circle is closed if and only if $|k|>2$. 
\end{proposition}
\begin{proof}[Sketch of proof.]
We solve the following equation:
\begin{equation}\label{C1}
\frac
{x^{\prime}y^{\prime \prime}-x^{\prime \prime}y^{\prime}}
{4y^2}+\frac{x^{\prime}}{y}=k,
\end{equation}
for nonzero constant $k$.
Setting
\begin{equation*}
X:=\frac{x^{\prime}}{2y}, \ \
Y:=\frac{y^{\prime}}{2y},
\end{equation*}
the equation \eqref{C1} becomes
\begin{equation*}
\label{C3}
k=XY^{\prime}-X^{\prime}Y+2X.
\end{equation*}
Since $\beta$ is parametrized by arclength parameter, then
$X$ and $Y$ satisfies $X^2+Y^2=1$.
Consequently, we have to solve the following system
\begin{equation*}
\label{C4}
X^2+Y^2=1, \ \
XY^{\prime}-X^{\prime}Y+2X=k. 
\end{equation*}
To do this, we introduce the function $\mu=\mu(s)$, as follows
\begin{equation*}
\label{C5}
X=\cos\mu(s), \ \
Y=\sin\mu(s), 
\ \ \frac{d\mu}{ds}=k-2\cos\mu. 
\end{equation*}

In case $\mu$ is constant, we have 
$|k|\leq 2$ and 
$\sin\mu=\pm \sqrt{4-k^2}/2$.

If $k^2=4$, then $Y=0$ and hence $y$ is constant. 
Thus $\beta$ is a horizontal line. 

If $k^2\not=4$, then the system 
\begin{align}\label{C7}
\frac{dx}{ds}=&2y\cos \mu,
\\
\frac{dy}{ds}=&2y\sin \mu
\label{C8}
\end{align}
is solved as 
\[
x=(\cot\mu)y+x_0.
\]
Thus $\beta$ is an Euclidean line and it is usually known as
a \emph{hypercycle} or and \emph{equidistant line}. Of course, the particular
situation $\cos\mu=0$ leads to a vertical (half) line, which is a geodesic.

Next we consider the case $\mu$ is non-constant.
Then the derivative of $x$ and $y$ 
are given by
\begin{equation*}
\frac{dx}{ds}
=\frac{dx}{d\mu}(k-2\cos\mu),\quad 
\frac{dy}{ds}
=\frac{dy}{d\mu}(k-2\cos\mu).
\end{equation*}
From these, we get
\begin{equation}\label{C9}
\frac{dy}{y}
=\frac{2\sin\mu}{k-2\cos\mu}d\mu.
\end{equation}
Solving \eqref{C9}, we get
\begin{equation}\label{C10}
y=r(k-2\cos\mu),\ r > 0.
\end{equation}
Next, inserting
\eqref{C10} to \eqref{C7}, we have
$dx=2r\cos \mu \>d\mu$, which implies
\begin{equation*}
x=2r\sin \mu+a,\ a \in \mathbb{R}.
\end{equation*}

Hence the Riemannian circle of curvature $\kappa_\beta\not=0$ is a horizontal line, 
an equidistant line or a part of the Euclidean circle:
\[
(x-a)^2+(y-rk)^2=4r^2.
\]
It is straightforward that $\beta$ is closed if and only
if the circle lies entirely above the boundary line, equivalently to $|k|>2$.
Furthermore, $\beta$ is a horocycle if and only if
$|k|=2$.
When $|k|<2$, the curve $\beta$ is a portion of an Euclidean circle that makes non-right
angles with the boundary line.

\end{proof}

{\bf Conclusion.}
We investigate periodic contact magnetic curves in $\mathrm{SL}_2\mathbb{R}$ and give a criterion
to have periodicity. In such a way, we obtain a quantization principle for periodic contact magnetic 
curves in $\mathrm{SL}_2\mathbb{R}$ over the set of rational numbers.
This conclusion is similar to that for closed geodesics on a torus $\mathbb{T}^2$ and has a physical
meaning: every closed geodesic corresponds to a discrete set of energy levels, 
"mirroring the analogous quantization of energy levels in the model of an atom" \cite{Jan12}.
On the other hand, it is proved \cite{Iriyeh, SW} that every closed L-minimal Legendre curve in the
$3$-sphere $\mathbb{S}^3$ is a magnetic curve. These curves are known as torus knots
and Kajigaya \cite{Kajigaya} proved that they are L-unstable.

{\bf Acknowledgments.} 
This work was partially supported by Kakenhi 15K04834 (Japan).  
The first named author wishes to thank 
'Gheorghe Asachi' Technical University of Iasi and 
University 'Alexandru Ioan Cuza' of Iasi
for warm hospitality he received during his visit in November 2017.



\begin{thebibliography}{99}

\bibitem{Av}
A.~Arvanitoyeorgos,
\emph{Homogeneous manifolds whose geodesics are orbits. Recent results and some open problems},
Irish Math. Soc. Bull. {\bf 79} (2017) 5--29 (see also \texttt{arXiv:1706.09618}).


\bibitem{Bar08} M.~Barros, 
\emph{Simple geometrical models with applications in Physics}, 
CP1002, Curvature and Variational Modelling in Physics and Biophysics, 
Eds. O.J. Garay, E. Garcia-Rio and R. V\'{a}zquez-Lorenzo, AIP 2008, 71--113.
	
\bibitem{BCFR07} M.~Barros, J.~L.~Cabrerizo, M.~Fern\'{a}ndez, and A.~Romero,
    \emph{Magnetic vortex filament flows},
    J. Math. Phys. {\bf 48} (2007) 8, 082904.
    
\bibitem{BarrosRomero} M.~Barros, A.~Romero,
    \emph{Magnetic vortices},
    EPL {\bf 77} (2007), 34002:1--5.


\bibitem{BRCF05} M.~Barros, A.~Romero, J.~L.~Cabrerizo, and M.~Fern\'{a}ndez, 
	\emph{The Gauss-Landau-Hall problem on Riemannian surfaces}, 
	J. Math. Phys. {\bf 46} (2005), 112905.

	
\bibitem{Blair} D.~E.~Blair,  
			\emph{Riemannian Geometry of Contact and Symplectic Manifolds}, 
			Progress in Math. 203, 2002, Birkh\"auser, Boston-Basel-Berlin.
    
\bibitem{CFG09} J.~L.~Cabrerizo,  M.~Fern\'andez, and J.~S.~G\'omez,
    \emph{The contact magnetic flow in $3D$ Sasakian manifolds},
    J. Phys. A: Math. Theor. {\bf 42} (2009), 19, 195201.
		

\bibitem{jm:Com87} A.~Comtet,
    {\em On the Landau levels on the hyperbolic plane},
    Ann. Physics, {\bf 173} (1987) 1, 185--209.

\bibitem{DeTG2008} D.~De Turck and H.~Gluck,
\emph{Electrodynamics and the Gauss linking integral on the 3-sphere and in hyperbolic 3-space}, 
J. Math. Phys. {\bf 49} (2008), 023504.

\bibitem{DeTG2013} D.~De Turck and H.~Gluck,
\emph{Linking, twisting, writhing, and helicity on the 3-sphere and in hyperbolic 3-space}, 
J. Differential Geom. {\bf 94} (2013), 87--128.

\bibitem{DIMN1} S.~L.~Dru\c t\u a-Romaniuc, J.~Inoguchi, M.~I.~Munteanu, and A.~I.~Nistor,
	\emph{Magnetic curves in Sasakian manifolds}, J. Nonlinear Math. Phys. 
	{\bf 22} (2015), no.~3, 428--447.

\bibitem{DIMN2} S.~L.~Dru\c t\u a-Romaniuc, J.~Inoguchi, M.~I.~Munteanu, and A.~I.~Nistor,
	\emph{Magnetic curves in cosymplectic manifolds}, Reports Math. Phys. {\bf 78} (2016) 1, 33--48.

\bibitem{GM} B.~Gabrovsek, M.~Mroczkowski, 
\emph{Knots in the solid torus up to 6 crossing}, 
J. Knot Theory and Its Ramifications \textbf{21} (2012), no.~11, 1250106 (43 pages).

\bibitem{I2004}
J.~Inoguchi, 
\emph{Invariant minimal surfaces in 
the real special linear group of 
degree $2$},
Italian J. Pure Appl. Math. 
\textbf{16} (2004), 61--80. 
\texttt{arXiv:0910.3104}

\bibitem{I2004CM}
J.~Inoguchi, 
\emph{Submanifolds with harmonic mean curvature vector 
field in contact $3$-manifolds},
Coll. Math. 
\textbf{100} (2004), no.~2, 163--179.

\bibitem{IKOS3}
J.~Inoguchi, T.~Kumamoto,
N.~Ohsugi and
Y.~Suyama,
\emph{Differential geometry of curves and surfaces in 
$3$-dimensional homogeneous spaces 
$\mathrm{I}\!\mathrm{I}\!\mathrm{I}$},
Fukuoka Univ. Sci. Rep.
{\bf 30} (2000),
131--160.

\bibitem{IL}
J.~Inoguchi and J.-E.~Lee, 
\emph{
Slant curves in  $3$-dimensional almost contact metric geometry}, 
Internat. Elect. 
J. Geom. \textbf{8} (2015), 
no.~2, 106--146.

\bibitem{IM1} J.~Inoguchi and M.~I.~Munteanu, 
	\emph{Periodic magnetic curves in Berger spheres}, T{\^o}hoku Math. J. 
	{\bf 69} (2017), no.~1, 113--128.


\bibitem{Iriyeh} H.~Iriyeh, 
\emph{Hamiltonian minimal Lagrangian cone in 
$\mathbf{C}^m$}, Tokyo J. Math. {\bf 28} 
(2005), 1, 91--107.

	
\bibitem{Jan12} R.~T.~Jantzen, 
\emph{Geodesics on the torus and other surfaces of revolution clarified using
undergraduate physics tricks with bonus: nonrelativistic and relativistic Kepler problems},
arXiv:1212.6206v1 [math.DG] (2012)


\bibitem{Kajigaya}
T.~Kajigaya, 
\emph{Second variational formula and the stability of 
Legendrian minimal submanifolds in Sasakian
manifolds}, T{\^o}hoku Math. J. {\bf 65} (2013), 4, 523--543.

\bibitem{Kokubu}
M.~Kokubu,
\emph{On minimal surfaces in the real
special linear group}
$SL(2,{\mathbf R})$,
Tokyo J. Math.
{\bf 20} (1997), no.~2, 
287--297.


\bibitem{KV}
O.~Kowalski, L.~Vanhecke, 
\emph{Riemannian manifolds with homogeneous geodesics}, 
Boll. Un. Math. Ital. B (7) \textbf{5} (1991), no.~1, 189--246.

\bibitem{Nistor}
A.~I.~Nistor, 
\emph{Motion of charged particles in a Killing magnetic field in 
$\mathbb{H}^2\times\mathbb{R}$}, 
Rend. Sem. Mat. Univ.
Politec. Torino \textbf{73/1} (2016), 
no.~3-4, 161--170 

\bibitem{OR}
C.~Oberti, 
R.~L.~Ricca, 
\emph{Energy and helicity of magnetic torus knots and braid},
Fluid Dyn. Res. {\bf 50} (2018) 011413 (13pp).

\bibitem{Profir}	
M.~M.~Profir, 
\emph{Sugli spazi omogenei di dimensione tre $SO(2)$-isotropi},
Dottorato di Ricerca in Math., Univ. Cagliari, 2008.

\bibitem{RN2011}
R.~L.~Ricca, B.~Nipoti,
\emph{Gauss' linking number revisited}, 
J. Knot Theory and Its Ramifications 
\textbf{20} (2011), no.~10, 1325--1343.


\bibitem{SW}
R.~Schoen and J.~Wolfson, 
\emph{Minimizing area among Lagrangian surfaces: 
The mapping problem}, J. Differential Geom., {\bf 58} (2001), 1, 1--86.	


\bibitem{Sitzia}	
P.~Sitzia, 
\emph{Explicit formulas for geodesics of three-dimensional homogeneous manifolds with
isometry group of dimension $4$ or $6$}, preprint, Univ. Cagliari, 1989.
	
	
\bibitem{jm:Sun93} T.~Sunada,
    {\em Magnetic flows on a Riemann surface},
    Proceedings of KAIST Mathematics Workshop, 1993, 93--108.

\bibitem{Taub07} C.~H.~Taubes, 
\emph{The Seiberg-Witten equations and the Weinstein conjecture}, 
Geom. Topol. {\bf 11} (2007) 4, 2117--2202.


\end{thebibliography}
\end{document}